%% file: Chebyshev_bias.tex
\documentclass{amsart}

\usepackage[latin1]{inputenc}
\usepackage{amsthm}
\usepackage{amsmath}
\usepackage{amssymb}
\usepackage{mathrsfs}
\usepackage{enumitem}

\usepackage[T1]{fontenc}

\usepackage{tikz}
\usetikzlibrary{matrix,arrows,decorations.pathmorphing}

\usepackage[hidelinks]{hyperref}
\usepackage[left=2.3cm, right=2.3cm]{geometry}

\theoremstyle{plain}
\newtheorem{theo}{Theorem}[section]
\newtheorem{prop}[theo]{Proposition}
\newtheorem{lem}[theo]{Lemma}
\newtheorem{cor}[theo]{Corollary}

\theoremstyle{definition}
\newtheorem{conj}[theo]{Conjecture}
\newtheorem{defi}[theo]{Definition}

\theoremstyle{remark}
\newtheorem{Rk}{Remark}
\newtheorem{ex}{Example}

\DeclareMathOperator{\im}{Im}
\DeclareMathOperator{\re}{Re}
\DeclareMathOperator{\Sym}{Sym}
\DeclareMathOperator{\Var}{Var}
\DeclareMathOperator{\Li}{Li}

\newcommand*\diff{\mathop{}\!\mathrm{d}}
\def\sums{\mathop{\sum \Bigl.^{*}}\limits}

\begin{document}

\title{Chebyshev's bias for analytic $L$-functions}
\author{Lucile Devin}
\email{ldevin2@uottawa.ca}
\address{Laboratoire de Math{\'e}matiques d'Orsay \\ Univ. Paris-Sud \\ CNRS \\ Universit{\'e} Paris-Saclay \\ 91405 Orsay  \\ France}
\curraddr{Department of Mathematics and Statistics \\ University of Ottawa \\ 585 King Edward ave \\ Ottawa \\ Ontario K1N 6N5 \\ Canada}

\keywords{Chebyshev's bias, L-functions}
\subjclass[2010]{Primary 11N45, 11F30, 11S40; Secondary 11G40}
\date\today

\begin{abstract}
	In this paper we discuss the generalizations of the concept of Chebyshev's bias from two perspectives.
	First we give a general framework for the study of prime number races and Chebyshev's bias attached to general $L$-functions satisfying natural analytic hypotheses.
	This extends the cases previously considered by several authors and involving, among others, Dirichlet $L$-functions and Hasse--Weil $L$-functions of elliptic curves over $\mathbf{Q}$. 
	This also apply to new Chebyshev's bias phenomena that were beyond the reach of the previously known cases. 
	In addition we weaken the required hypotheses such as GRH or linear independence properties of zeros of $L$-functions.
	In particular we establish the existence of the logarithmic density of the set 
	$\lbrace x\geq 2 : \sum_{p\leq x} \lambda_{f}(p) \geq 0 \rbrace$ 
	for coefficients $(\lambda_{f}(p))$ of general $L$-functions conditionally on a much weaker hypothesis than was previously known. 
\end{abstract}

\maketitle

\section{Introduction}

\input{intro_bias}

\input{TheoremExamples}

\section{Applications to old and new instances of prime number races}\label{Sec_Applications}

\input{Applications}

\section{Proof of Theorem \ref{Th_DistLim}}\label{Sec_ProofExist}

\input{Proof_existence}

\section{Results under additional hypotheses}\label{Sec_AddHyp}

\input{AddHyp}

\section*{Acknowledgements}
This paper contains some of the results of my doctoral dissertation \cite{L_These}.
I am very grateful to my advisor Florent Jouve for suggesting the problem, for all his advice, help and time spent correcting the first drafts of this paper.
I thank Daniel Fiorilli and \'{E}tienne Fouvry for their interest in this work and for the many conversations that have led to improvements in the results.
This work has also benefited from conversations with Farrell Brumley, Ga{\"e}tan Chenevier, Andrew Granville, Emmanuel Kowalski, Greg Martin, Philippe Michel, Nathan Ng, Ze{\'e}v Rudnick and Miko\l{}aj Fr\k{a}czyk.
I thank Jean-Pierre Serre for his interest in this work and for providing the proof of Lemma~\ref{lem_NonConstant}. 
The computations presented in this document have been performed with the \texttt{SageMath} \cite{sagemath} software.

\bibliographystyle{amsalpha} 
\bibliography{biblio}



\end{document}

%% file: intro_bias.tex
\subsection{Context}

In 1853 Chebyshev noticed in a letter to Fuss that 
there is a bias in the distribution of primes modulo $4$.
In initial intervals of the integers, 
there seems to be more primes congruent to $3\ [\bmod\ 4]$ than congruent to $1\ [\bmod\ 4]$.
Over the years the synonymous expression ``prime number races'' has emerged to describe problems of Cheyshev's type.
Since then, it has been quite investigated and generalized in other number theoretical contexts.

In \cite{RS} Rubinstein and Sarnak gave a framework for the quantification of Chebyshev's bias in prime number races in arithmetic progressions.
Following an observation of Wintner \cite{Wintner} used for the race between $\pi(x)$ and $\Li(x)$, 
they studied the logarithmic density of the set
$ \lbrace x \geq 2 : \pi(x; q,a) >\pi(x;q, b) \rbrace $
where $\pi(x; q,a)$ is the number of primes $\leq x$ that are congruent to $a\ [\bmod\ q]$.
Under the Generalized Riemann Hypothesis (GRH) and Linear Independence (LI for short) of the zeros, 
Rubinstein and Sarnak answered Chebyshev's original question.
Precisely they showed that the logarithmic density $\delta(4;3,1)$ of the set of $x\geq2$ for which $\pi(x; 4,3)>\pi(x;4, 1)$
exists and is about $0.9959$.
In their paper about the Shanks--Renyi prime number race \cite{FiorilliMartin}, 
Fiorilli and Martin gave a more precise estimation for $\delta(q;a,b)$ in general under the same hypotheses.
For more details on prime number races, we refer to the expository article of Granville and Martin \cite{GranvilleMartin}, see also \cite{FordKonyagin}.

The method used by Rubinstein and Sarnak is to prove conditionally the existence of a 
limiting logarithmic distribution for the vector valued function encoding the prime number race. 
This logarithmic distribution plays a crucial role in the analysis of the bias. It has since then been generalized greatly to study variants of Chebyshev's question coming from a broad variety of arithmetic contexts. 
Let us quickly review some of them. 

In his thesis \cite{NgThesis}, Ng generalized this question to that
of biases in the distribution of Frobenius substitutions in conjugacy classes of Galois groups of number fields. 
Here the underlying equidistribution property is Chebotarev's density Theorem.
Ng's results are conditional on GRH, LI and Artin's Conjecture for Artin $L$-functions.

In his expository paper about error terms in arithmetic, 
Mazur \cite{MazurErrorTerm} raised the question of prime number races for elliptic curves, 
or more generally for the Fourier coefficients of a modular form.
For example he plotted graphs of functions:
\begin{align*}
x\mapsto \lvert\lbrace p\leq x : a_{p}(E) >0 \rbrace\rvert -  \lvert\lbrace p\leq x : a_{p}(E) <0 \rbrace\rvert
\end{align*}
where $a_{p}(E) = p+1 - \lvert E(\mathbf{F}_{p})\rvert$, 
for some elliptic curve $E$ defined over $\mathbf{Q}$.
He observed that the race between the primes such that $a_{p}(E) >0 $ 
and the primes such that $a_{p}(E) <0 $
tends to be biased towards negative values when the algebraic rank of the elliptic curve is large.
In a letter to Mazur \cite{SarnakLetter}, 
Sarnak gave an effective framework to answer Mazur's question.
Under GRH and LI, 
he explained the prime number race using the zeros of all the symmetric powers $L(\Sym^{n}E,s)$ of the Hasse--Weil $L$-function of $E/\mathbf{Q}$.
Sarnak also introduced a related (simpler) race: study the sign of the function 
$\sum_{p\leq x}\frac{a_{p}(E)}{\sqrt{p}}$.
In \cite{FioEC}, Fiorilli developed Sarnak's idea
and gave sufficient conditions to get highly biased prime number races in the context of elliptic curves conditionally on weaker versions of GRH and LI.

More recently, Akbary, Ng and Shahabi (\cite{ANS}) used the theory of almost periodic functions to study the limiting distribution associated to a very wide range of $L$-functions.

\medskip

In this paper, we generalize the questions above to prime number races for the coefficients of \emph{analytic $L$-functions}. 
We prove unconditionally (Theorem \ref{Th_DistLim}) the existence of the limiting logarithmic distribution associated to the prime number race for a wide variety of usual $L$-functions  including Dirichlet $L$-functions and Hasse--Weil $L$-functions.
In particular we obtain unconditional proofs of some of the results of \cite{RS}.
Our general framework is also applicable to new instances of Chebyshev's bias phenomena.
For example we prove unconditionally that, after suitable scaling the functions
\begin{align*}
x\mapsto\sum_{p=a^{2} + 4b^{2}\leq x}\frac{a^{2} - 4b^{2}}{a^{2} + 4b^{2}}, \qquad  
x\mapsto\sum_{p=a^{2} + b^{2}\leq x}\frac{a^{4} + b^{4} - 6a^2b^2}{(a^{2} + b^{2})^2}, \qquad
x\mapsto\sum_{p\leq x}\frac{a_{p}(E_{i})a_{p}(E_{j})}{p}
\end{align*}
admit limiting logarithmic distributions with negative average value (see Theorem~\ref{Th_sum2squares}, \ref{Th_BiasGauss} and \ref{Prop_CorrEllcurves} for precise statements). 
For the first function, the fact that the average value is negative gives evidence (see Corollary~\ref{Cor_mean_bias}) that when writing $p=a^{2}+4b^{2}$ there is a bias: the even square tends to be more often larger than the odd square.

We also study minimal conditions to ensure that the distribution has nice properties such as regularity, symmetry, and concentration. We obtain results comparable to \cite{RS} under weaker hypotheses.
In particular, we prove the existence of the logarithmic density $\delta(4;3,1)$ conditionally on a weaker version of LI related to the notion of self-sufficient zeros introduced by Martin and Ng \cite{MartinNg}, (see Theorem~\ref{Th_SomeSelfSufficience}). We also highlight a relation between the support of the distribution and Riemann Hypothesis (see Theorem~\ref{Th_withRH}).

%% file: TheoremExamples.tex
\subsection{Setting}\label{subsec_setting}

In the present paper we use a custom-made definition of ``analytic'' $L$-function inspired by \cite[Chap. 5]{IK} and Selberg's class.
We will only use analytic properties of the function to study the associated prime number race.

\begin{defi}[Analytic $L$-function]\label{Def_Lfunc}
Let $L(f,s)$ be a complex-valued function of the variable $s\in\mathbf{C}$ attached to an auxiliary parameter $f$ to which one can attach an integer $q(f)$ (usually $f$ is of arithmetic origin and $q(f)$ is its conductor). 
We say that $L(f,s)$ is an analytic $L$-function if we have the following data and conditions:
\begin{enumerate}
\item\label{Hyp_Euler+RamanujanPetersson} A Dirichlet series factorizing as an Euler product of degree $d\geq 1$ that coincides with $L(f,s)$ for $\re(s)>1$:
\begin{align*}
L(f,s) = \sum_{n\geq 1} \lambda_{f}(n)n^{-s} = \prod_{p}\prod_{j=1}^{d}\left(1-\alpha_{f,j}(p)p^{-s}\right)^{-1}
\end{align*}
with $\lambda_{f}(1)=1$ and $\alpha_{f,j}(p) \in \mathbf{C}$,
satisfying $\lvert \alpha_{f,j}(p)\rvert=1$ for all $j$ and $p\nmid q(f)$.
In particular the series and Euler product are absolutely convergent for $\re(s)>1$.
\item\label{Hyp_FunctEquation} A gamma factor with local parameters $\kappa_{j}\in\mathbf{C}$, $\re(\kappa_{j})>-1$:
\begin{align*}
\gamma(f,s) = \pi^{-ds/2}\prod_{j=1}^{d}\Gamma\left(\frac{s + \kappa_{j}}{2}\right).
\end{align*}
The \emph{analytic conductor} of $f$ is then defined by: 
\begin{align*}
\mathfrak{q}(f) = q(f)\prod_{j=1}^{d}\left( \lvert \kappa_{j}\rvert + 3\right),
\end{align*}
and we can define the completed $L$-function
\begin{align*}
\Lambda(f,s) =  q(f)^{s/2}\gamma(f,s)L(f,s).
\end{align*}
It admits an analytic continuation to a meromorphic function of order $1$, with at most poles at $s=0$ and $s=1$.
Moreover it satisfies a functional equation $\Lambda(f,s)=\epsilon(f)\Lambda(\overline{f},1-s)$, 
with $\lvert\epsilon(f)\rvert =1$.
Here $\Lambda(\overline{f},1-s)$ is the completed $L$-function associated to 
$L(\overline{f},s) := \sum_{n\geq 1} \overline{\lambda_{f}(n)}n^{-s}$.
\item\label{Hyp_L2isLfunct} The \emph{second moment $L$-function}
$$L(f^{(2)},s) = \prod_{p}\prod_{j=1}^{d}\left(1-\alpha_{j}(p)^{2}p^{-s}\right)^{-1}$$
is defined for $\re(s)>1$.
We assume that there exists an open subset 
$U \supset \lbrace\re(s)\geq 1\rbrace$ such that
$L(f^{(2)},\cdot)$ can be continued to a meromorphic function for $s\in U$,
and on $U-\lbrace 1\rbrace$ there is neither a zero nor a pole of $L(f^{(2)},\cdot)$.
\end{enumerate}

\end{defi}

\begin{Rk}
	\begin{enumerate}
		\item {[Order of a meromorphic function].}
		A meromorphic function on $\mathbf{C}$ is said to be of order $1$ if it can be written as the quotient of two entire functions of order $1$ (that is functions $\Lambda$ such that for every $\beta>1$ and no $\beta<1$ one has
		$\lvert \Lambda(s) \rvert\ll \exp(\lvert s\rvert^{\beta})$).
		Usually this property is obtained by proving that the function is bounded in vertical strips.
		This occurs naturally in the proof of the functional equation using the method of zeta-integrals 
		(see e.g. \cite[Cor. 13.8, Prop. 13.9]{GodementJacquet} for the general case of automorphic $L$-functions).
		
		In particular by Jensen's formula (see \cite[15.20]{Rudin_Analyse}) one can prove that the sum over the zeros 
		$\sum\limits_{\Lambda(\rho) =0}\frac{1}{\lvert \rho\rvert^{1+\epsilon}}$
		converges for every $\epsilon>0$.
		
		\item {[Second moment].}
		In \cite{Conrad}, the function $L(f^{(2)},s)$ used in Definition~\ref{Def_Lfunc}.(\ref{Hyp_L2isLfunct}) is called the \emph{second moment} of $L(f,s)$ over $\mathbf{Q}$.
		We note that it is determined by the local roots $\alpha_{j}(p)$ over $\mathbf{Q}$ 
		(rather than by the meromorphic function $L(f,s)$)
		and it is related to the Rankin-Selberg product (see Example~\ref{Ex_Lfunc}.(\ref{Ex_Lfunc_cusp})).
		The assumption on the function $L(f^{(2)},s)$ is the \emph{second moment hypothesis} (\cite[Def. 4.4]{Conrad}).
	\end{enumerate}
	
\end{Rk}

\begin{ex}\label{Ex_Lfunc}
	\begin{enumerate}
	    \item\label{Ex_Lfunc_zeta} Riemman's zeta function is a classical example of $L$-function, it satisfies all the conditions of Definition~\ref{Def_Lfunc}, see Theorem~\ref{Th_race_pi_Li}.
		\item\label{Ex_Lfunc_Dirichlet} Dirichlet $L$-functions are analytic $L$-functions. Hence some of the results of \cite{RS} can be given unconditionally, see Theorem \ref{Th_RS_raceModq}.		
		\item\label{Ex_Lfunc_Modular2} Modular $L$-functions of degree $2$ are analytic $L$-functions.
		It is a consequence of results of Deligne and Serre \cite[Th. 8.2]{Deligne_Weil1} and \cite{DeligneSerre}.
		In particular, following results on modularity (\cite{Wiles_ModEll,TaylorWiles_Modularity,BCDT_Modularity}), 
		if $E/\mathbf{Q}$ is an elliptic curve, $L(E,s)$ is an analytic $L$-function.
		This property was already used by Fiorilli in \cite{FioEC}, see Proposition \ref{Prop_Fi_EC}.		
		\item\label{Ex_Lfunc_cusp} Under the Ramanujan--Petersson Conjecture, general automorphic $L$-functions associated to cusp forms on $GL(m)$ for $m\geq 1$, are analytic $L$-functions in the sense of Definition \ref{Def_Lfunc}.
		Indeed (\ref{Hyp_Euler+RamanujanPetersson}) precisely says that the Ramanujan--Petersson Conjecture is satisfied
		and (\ref{Hyp_FunctEquation}) is known for such $L$-functions (\cite{GodementJacquet}, \cite{Cogdell}).
		One has
		\begin{align*}
		L(f^{(2)},s) = L(\Sym^{2}f,s)L(\wedge^{2}f,s)^{-1} =L(f\otimes f,s)L(\wedge^{2}f,s)^{-2}.
		\end{align*}
		By \cite[Th. 6.1, Th. 7.5]{BumpGinzburg} for $L(\Sym^{2}f,s)$ and \cite[Th. 1-3]{BumpFriedberg}, \cite[Th. 1-2]{JacquetShalika} for $L(\wedge^{2}f,s)$,
there exists an open subset 
$U \supset \lbrace\re(s)\geq 1\rbrace$ such that
these two functions admit a meromorphic continuation to $s\in U$,
and on $U-\lbrace 1\rbrace$ they have neither a zero nor a pole, hence (\ref{Hyp_L2isLfunct}) is satisfied. 
		As F. Brumley pointed out to us one could ask for the third hypothesis above to be about any two of the three functions $L(f\otimes f,s)$, $L(\Sym^{2}f,s)$ and $L(\wedge^{2}f,s)$.
		
		In \cite[Cor. 1.5]{ANS}, the existence of the limiting logarithmic distribution 
		for the function $\psi(f,x)/\sqrt{x}$ associated to an automorphic $L$-function $L(f,s)$ is proved under GRH 
		and does not depend on the Ramanujan--Petersson Conjecture.
		In the present paper we need to assume the Ramanujan--Petersson Conjecture but not GRH to prove that 
		the function $ \pi(f,x)\log x/x^{\beta_{f,0}}$ 
		has a limiting logarithmic distribution (see Theorem~\ref{Th_Aut_bias}). 
		
		\item\label{Ex_Lfunc_RankinSelberg} If $L(f,s)$ and $L(g,s)$ are two modular $L$-functions of degree $2$, such that $g\neq \overline{f}$, 
		then the  Rankin-Selberg product $L(f\otimes g,s)$ is an analytic $L$-function.
		The conditions (\ref{Hyp_Euler+RamanujanPetersson}) and (\ref{Hyp_FunctEquation}) are satisfied (see e.g. \cite[Th.2.3]{CogdellPS}).
		One has 
		\begin{align*}
        L((f\otimes g)^{(2)},s) = \frac{L(\Sym^{2}f\otimes \Sym^{2}g,s) 
        L(\chi_{f}\chi_{g},s)}{L((\Sym^{2}f)\otimes      \chi_{g},s)L((\Sym^{2}g)\otimes\chi_{f},s)}
        \end{align*}
        where $\chi_{f}$, $\chi_{g}$ are the nebentypus respectively associated to the modular forms $f$ and $g$.
		We deduce that Condition (\ref{Hyp_L2isLfunct}) of Definition~\ref{Def_Lfunc} is satisfied.
		We use this property in Theorem \ref{Prop_CorrEllcurves} in the case $f$ and $g$ are associated to two elliptic curves over $\mathbf{Q}$ that are non-isogeneous in a strong sense.
	\end{enumerate}
	
\end{ex}

The parameters in Definition~\ref{Def_Lfunc} satisfy
for every prime $p$, $\lambda_{f}(p)=\sum_{j=1}^{d}\alpha_{f,j}(p)$.
In case $L(f,\cdot)=L(\overline{f},\cdot)$ is real, one has $\lambda_{f}(p)\in\mathbf{R}$.
The Generalized Sato--Tate conjecture states that in this case the $(\lambda_{f}(p))_{p}$ should equidistribute 
according to a certain probability measure on $[-d,d]$.
In case the $L$-function is entire and does not vanish on the line $\re(s) =1$, 
a general Prime Number Theorem for the $L$-function implies 
that the Sato--Tate law has mean value equal to $0$.
We expect this to hold more generally.

\begin{conj}\label{Sato-Tate} 
Let $\mathcal{S}$ be a finite set of entire analytic $L$-functions. 
Suppose $\mathcal{S}$ is stable by conjugation (i.e. $\overline{\mathcal{S}}=\mathcal{S}$), 
and let $(a_{f})_{f\in \mathcal{S}}$ be a set of complex numbers satisfying $a_{\overline{f}}= \overline{a_{f}}$.
Then the sequence $(\sum_{f\in\mathcal{S}}a_{f}\lambda_{f}(p))_{p}$ is equidistributed in an interval of $\mathbf{R}$ 
according to a Sato--Tate law with mean value $0$.
\end{conj}

In particular for $L$-functions associated with an elliptic curve over $\mathbf{Q}$, 
Conjecture~\ref{Sato-Tate} is a (very) weak version of the Sato--Tate conjecture, and is known to hold (\cite{CHT_SatoTate}, \cite{HSBT_SatoTate}).

\subsection{Chebyshev type questions for analytic $L$-functions}

The Chebyshev type question this paper primarily focuses on is the following.
Let $\mathcal{S}$ be a finite set of entire analytic $L$-functions such that $\overline{\mathcal{S}}=\mathcal{S}$, 
and $(a_{f})_{f\in \mathcal{S}}$ a set of complex numbers satisfying $a_{\overline{f}}= \overline{a_{f}}$.
Under Conjecture \ref{Sato-Tate} one has 
$$\frac{1}{\pi(x)}\sum_{p\leq x}\sum_{f\in\mathcal{S}}a_{f}\lambda_{f}(p) \rightarrow 0$$ as $x\rightarrow +\infty$.
It is natural to study the sign of the summatory function
\begin{align}\label{Expr_Sum_function1}
x\mapsto \sum_{p\leq x}\sum_{f\in\mathcal{S}}a_{f}\lambda_{f}(p)
\end{align}
for $x>0$.

\begin{Rk}
In the case of elliptic curve $L$-functions (where the Sato--Tate law is known to be symmetric),
Mazur (\cite{MazurErrorTerm}) was first interested in studying the function
\begin{align*}
x \mapsto \lvert\lbrace p\leq x : \sum_{f\in\mathcal{S}}a_{f}\lambda_{f}(p) >0\rbrace\rvert - \lvert\lbrace p\leq x : \sum_{f\in\mathcal{S}}a_{f}\lambda_{f}(p) <0\rbrace\rvert
\end{align*}
for $x>0$.
But as Sarnak showed in \cite{SarnakLetter}, 
for this study, we need information about all symmetric powers of the $L$-functions involved.
It may yield non-converging infinite sums.
Hence following Sarnak's idea, we focus on the summatory function as in (\ref{Expr_Sum_function1}).
\end{Rk}

More generally if $L(f,\cdot)$ has a pole of order $r_{f}$ at $s=1$,
we study the sign of the summatory function
$$S:x\mapsto \sum_{p\leq x}\sum_{f\in\mathcal{S}}a_{f}\lambda_{f}(p)  - \sum_{f\in\mathcal{S}}a_{f}r_{f} \Li(x).$$
In \cite{SarnakLetter}, Sarnak presents a method to deal with this question in the case $\mathcal{S}$ is a singleton.
Precisely he considers the cases of the $L$-function associated to Ramanujan's $\tau$ function and of $L$-functions associated with an elliptic curve over $\mathbf{Q}$. 

Building on this method, we wish to understand the set of $x$ for which $S(x)\geq0$.
As Kaczorowski showed that in certain situations the natural density does not exists \cite{Kaczorowski}, we use the logarihmic density to measure this sets.

\begin{defi}\label{Def_logdens}
\begin{enumerate}
\item Define $$\overline{\delta}(\mathcal{S}) = \limsup\frac{1}{Y}\int_{2}^{Y}\mathbf{1}_{\geq 0}(S(e^{y}))\diff y \ \text{  and }
\ \underline{\delta}(\mathcal{S}) = \liminf\frac{1}{Y}\int_{2}^{Y}\mathbf{1}_{\geq0}(S(e^{y}))\diff y.$$
If these two densities are equal, we denote $\delta(\mathcal{S})$ their common value. These quantities measure the bias of $S(x)$ towards positive values. 
\item If $\delta(\mathcal{S})$ exists and is $> \frac{1}{2}$ 
we say that there is a bias towards positive values. 
If it is $<\frac{1}{2}$ we say that there is a bias towards negative values. 
\end{enumerate}
\end{defi}

Under GRH and LI, 
Sarnak (\cite{SarnakLetter}) showed that for an $L$-function of degree $2$, 
the bias exists and always differs from $\frac{1}{2}$.
One of the main results of this article is that the bias exists without assuming GRH and under a hypothesis weaker than LI on the independence of the zeros of the $L$-functions involved (see Theorem \ref{Th_SomeSelfSufficience}).
To state the existence of the bias we first need to prove that 
a suitable normalization of the function $S(x)$ admits a limiting logarithmic distribution.

\begin{defi}
Let $F:\mathbf{R}\rightarrow\mathbf{R}$ be a real function, 
we say that $F$ admits a limiting logarithmic distribution $\mu$ if
for any bounded Lipschitz continuous function $g$, we have
\begin{align*}
\lim_{Y\rightarrow\infty}\frac{1}{Y}\int_{2}^{Y}g(F(e^{y}))\diff y = 
\int_{\mathbf{R}}g(t)\diff\mu(t).
\end{align*}
\end{defi}

Before stating our main result we need to set some notation.
Since we do not assume the Riemann Hypothesis for our $L$-functions,
we denote
$$\beta_{\mathcal{S},0} =\sup\lbrace \re(\rho) : \exists f\in\mathcal{S},  L(f,\rho)=0\rbrace.$$ 
One has $\beta_{\mathcal{S},0}\geq \frac{1}{2}$ and equality is equivalent to the Riemann Hypothesis
for all the $L$-functions $L(f,\cdot)$, $f\in\mathcal{S}$.
Define
\begin{align*}
&\mathcal{Z}_{\mathcal{S}} = \lbrace \gamma>0 : \exists f \in \mathcal{S}, L(f,\beta_{\mathcal{S},0} + i\gamma)=0\rbrace,
&\mathcal{Z}_{\mathcal{S}}(T) = \mathcal{Z}_{\mathcal{S}}\cap(0,T],
\end{align*}
seen as multi-sets of zeros of largest real part (i.e. we count the zeros with multiplicities).
We denote $\mathcal{Z}_{\mathcal{S}}^{*}$ and $\mathcal{Z}_{\mathcal{S}}^{*}(T)$ the corresponding sets (i.e. repetitions are not allowed). 
Note that these sets can be empty if $\beta_{\mathcal{S},0} > \frac{1}{2}$, this will not be a problem.

If it does not lead to confusion we may omit the subscript $\mathcal{S}$. 
In the case the set $\mathcal{S}$ is a singleton $\lbrace f \rbrace$ and $a_{f}=1$, we will write $f$ in subscript instead of $\lbrace f\rbrace$.

For $L$ a meromorphic function in a neighbourhood of a point $\rho\in\mathbf{C}$,
let $m(L,\rho)$ be the multiplicity of the zero of $L$ at $s=\rho$.
(One has $m(L,\rho)=0$ if $L(\rho)\neq 0$, $m(L,\rho)>0$ if $L(\rho)= 0$ and $m(L,\rho)<0$ if $L$ has a pole at $s=\rho$.)

\section{Statement of the theoretical results}
\subsection{Limiting distribution}

Our first result is the existence of the limiting logarithmic distribution for the prime number races associated to analytic $L$-functions.
\begin{theo}\label{Th_DistLim}
Let $\lbrace L(f,\cdot) : f\in\mathcal{S} \rbrace$ be a finite set of analytic $L$-functions such that $\overline{\mathcal{S}}=\mathcal{S}$, 
and $(a_{f})_{f\in \mathcal{S}}$ a set of complex numbers satisfying $a_{\overline{f}}= \overline{a_{f}}$.
Define
$$E_{\mathcal{S}}(x)=\frac{\log x}{x^{\beta_{\mathcal{S},0}}}\left(
\sum_{p\leq x}\sum_{f\in\mathcal{S}}a_{f}\lambda_{f}(p) + \sum_{f\in\mathcal{S}}a_{f}m(L(f,\cdot),1) \Li(x) \right).$$

The function $E_{\mathcal{S}}(x)$ admits a limiting logarithmic distribution $\mu_{\mathcal{S}}$.
There exists a positive constant $C$ (depending on $\mathcal{S}$) such that one has 
$$\mu_{\mathcal{S}}(\mathbf{R}-[-A,A]) \ll \exp(-C\sqrt{A}).$$
Moreover let $X_{\mathcal{S}}$ be a random variable of law $\mu_{\mathcal{S}}$,
then the expected value of $X_{\mathcal{S}}$ is
	$$\mathbb{E}(X_{\mathcal{S}}) = m_{\mathcal{S}} := \sum_{f\in\mathcal{S}}a_{f}
	\left(m(L(f^{(2)},\cdot),1)\delta_{\beta_{\mathcal{S},0}=1/2} 
	-\beta_{\mathcal{S},0}^{-1}m(L(f,\cdot),\beta_{\mathcal{S},0})\right),$$ 
and its variance is
	$$\Var(X_{\mathcal{S}}) = 2\sum_{\gamma \in \mathcal{Z}_{\mathcal{S}}^{*}}\frac{\lvert M(\gamma)\rvert^{2}}{\beta_{\mathcal{S},0}^{2}+\gamma^{2}}$$
	where for $\gamma$ in $\mathcal{Z}_{\mathcal{S}}^{*}$, 
	$M(\gamma) = \sum_{f\in\mathcal{S}}a_{f}m(L(f,\cdot),\beta_{\mathcal{S},0}+i\gamma)$.
\end{theo}

\begin{Rk}\label{Rk_DistLim}
	\begin{enumerate}
\item This result generalizes \cite[Th 1.1]{RS} and \cite[Th. 5.1.2]{NgThesis} which are conditional on GRH 
and respectively deal with the cases of sets of Dirichlet $L$-functions and sets of Artin $L$-functions under Artin's Conjecture. 
Similar results are obtained under GRH in \cite{ANS} for the prime number race corresponding to the function $\psi$ associated to general $L$-functions. 
An unconditional proof is given in \cite{FioEC} in the case $\mathcal{S}$ is a singleton composed of one Hasse--Weil $L$-function.
The proof of Theorem~\ref{Th_DistLim}, in section \ref{Sec_ProofExist}, is essentially an adaptation of Fiorilli's proof to more general $L$-functions.
\item Note that there is no assumption made about the set $\mathcal{Z}_{\mathcal{S}}$. 
In particular the limiting logarithmic distribution exists even if the set $\mathcal{Z}_{\mathcal{S}}$ is empty 
(see Remark~\ref{Rk_onProp_LimOfDist}(\ref{It_NoZero})).
\item\label{Rk3_mean_bias} The sign of the expected value $\mathbb{E}(X_{\mathcal{S}})$ gives an idea of the kind of bias one should expect.
When it is non-zero, we conjecture that the bias is imposed by the sign of the expected value.
Conditionally on additional hypotheses we can prove this conjecture 
(see Corollary~\ref{Cor_mean_bias}).
\item More precisely Theorem~\ref{Th_DistLim} states that ``on average'' the coefficients $\lambda_{f}(p)$ of an entire analytic $L$-function are equal to $\frac{
	\beta_{f,0}m(L(f^{(2)},\cdot),1)\delta_{\beta_{f,0}=1/2} 
	-m(L(f,\cdot),\beta_{f,0})}{p^{1-\beta_{f,0}}}$.
Under GRH the bias is due to the second moment function $L(f^{(2)},s)$. 
A sum over squares of primes appears and cannot be considered as an error term if the function $L(f^{(2)},s)$ admits a zero or a pole at $s=1$
(see Section~\ref{Subsec_ApproxFin}).
As A. Granville pointed out to us, this phenomenon is related to another kind of bias in the Birch and Swinnerton-Dyer conjecture.
Using this second moment function in the case of elliptic curves over $\mathbf{Q}$, 
Goldfeld \cite{Goldfeld} showed that there is an unexpected factor $\sqrt{2}$ appearing in the asymptotics for
partial Euler products for the $L$-function at the central point.
This is due to the fact that in the case of a non-CM elliptic curve over $\mathbf{Q}$, the corresponding second moment function has a zero of order $1$ at $s=1$ (see also Proposition~\ref{Prop_Fi_EC}).
Goldfeld's result has been generalized by Conrad \cite[Th. 1.2]{Conrad} to general $L$-functions (quite similar to our analytic $L$-functions of Definition~\ref{Def_Lfunc}). 

\end{enumerate}

\end{Rk}

\subsection{Further properties under extra hypotheses}

Under additional hypotheses over the zeros of the $L$-functions, we can deduce properties of $\mu_{\mathcal{S}}$, and in turn results on the bias. 
This idea is developed in the following results.
A standard hypothesis about the set $\mathcal{Z}_{\mathcal{S}}$ is the Linear Independence hypothesis (LI), 
we show under a weaker hypothesis about linear independence that the logarithmic density $\delta(\mathcal{S})$ (see Definition~\ref{Def_logdens}) exists.

\begin{theo}\label{Th_SomeSelfSufficience}
	Suppose that there exists $N\geq 1$ and $\lambda_{1},\ldots,\lambda_{N} \in \mathcal{Z}_{\mathcal{S}}^{*}$ such that 
	\begin{align*}
	\langle\lambda_{1},\ldots,\lambda_{N}\rangle_{\mathbf{Q}} \cap \langle \mathcal{Z}_{\mathcal{S}}^{*} - \lbrace \lambda_{1},\ldots,\lambda_{N}  \rbrace \rangle_{\mathbf{Q}} = \lbrace 0 \rbrace
	\end{align*}
	where $\langle \cdot \rangle_{\mathbf{Q}}$ denotes the $\mathbf{Q}$-span of a set of real numbers.
	Then the distribution $\mu_{\mathcal{S}}$ is continuous (\textit{i.e.} $\mu_{\mathcal{S}}$ assigns zero mass to finite sets), and $\delta(\mathcal{S})$ exists.
\end{theo}

This theorem is proved in section~\ref{sub_Indep}.
In \cite{RS}, \cite{NgThesis}, \cite{ANS} and \cite{FioEC} the corresponding result is obtained under LI i.e. assuming that $\mathcal{Z}_{\mathcal{S}}$ is linearly independent over $\mathbf{Q}$.
Using the theory of almost periodic functions, Kaczorowski and Ramar\'e \cite[Th. 3]{KaczoRama} prove the existence of the logarithmic density of a comparable set in the general setting of the Selberg class, under the Riemann Hypothesis only.
We prove other results related to the smoothness of the limiting distribution $\mu_{\mathcal{S}}$ in section~\ref{sub_Indep} using the concept of self-sufficient zeros introduced by Martin and Ng in \cite{MartinNg} (see Definition~\ref{Def_selfsuff}).

We are also interested in the symmetry of the distribution $\mu_{\mathcal{S}}$.
We prove the following result conditionally on a weak conjecture of linear independence of the zeros. 

\begin{theo}\label{Th_Indep_Sym}
Suppose that the set $\mathcal{S}$ satisfies the conditions of Theorem~\ref{Th_DistLim}.
Suppose that for every $(k_{\gamma})_{\gamma} \in \mathbf{Z}^{(\mathcal{Z}_{\mathcal{S}})}$ one has
$$ \sum_{\gamma \in \mathcal{Z}_{\mathcal{S}}}k_{\gamma}\gamma =0 \Rightarrow \sum_{\gamma \in \mathcal{Z}_{\mathcal{S}}}k_{\gamma}\equiv 0\ [\bmod\ 2].$$
Then the distribution $\mu_{\mathcal{S}}$ is symmetric with respect to $m_{\mathcal{S}}$.
\end{theo}

We prove Theorem~\ref{Th_Indep_Sym} in Section~\ref{sub_Sym}.
This theorem improves again a result obtained in \cite{RS} under LI.

We can now come back to Remark~\ref{Rk_DistLim}(\ref{Rk3_mean_bias}). 
If the bias exists, it should be imposed by the sign of the average value of the limiting distribution.
We get the following result as a corollary of Theorems~\ref{Th_DistLim}, \ref{Th_SomeSelfSufficience}, and Theorem~\ref{Th_Indep_Sym} or Chebyshev's inequality (Lemma~\ref{Lem_ChebyshevIneq}).

\begin{cor}\label{Cor_mean_bias}
	Let $\lbrace L(f,\cdot) : f\in\mathcal{S} \rbrace$ be a finite set of analytic $L$-functions such that $\overline{\mathcal{S}}=\mathcal{S}$, 
	and $(a_{f})_{f\in \mathcal{S}}$ a set of complex numbers satisfying $a_{\overline{f}}= \overline{a_{f}}$.
	Suppose that:
	\begin{enumerate}
		\item\label{Hyp_biasExist} there exists $N\geq 1$ and $\lambda_{1},\ldots,\lambda_{N} \in \mathcal{Z}_{\mathcal{S}}^{*}$ such that 
		$
		\langle\lambda_{1},\ldots,\lambda_{N}\rangle_{\mathbf{Q}} \cap \langle \mathcal{Z}_{\mathcal{S}}^{*} - \lbrace \lambda_{1},\ldots,\lambda_{N}  \rbrace \rangle_{\mathbf{Q}} = \lbrace 0 \rbrace,
		$
		\item[(2.a)]\label{Hyp_symm} for every $(k_{\gamma})_{\gamma} \in \mathbf{Z}^{(\mathcal{Z}_{\mathcal{S}})}$ one has
		$ \sum_{\gamma \in \mathcal{Z}_{\mathcal{S}}}k_{\gamma}\gamma =0 \Rightarrow \sum_{\gamma \in \mathcal{Z}_{\mathcal{S}}}k_{\gamma}\equiv 0\ [\bmod\ 2],$  
		\item[\emph{or}] 
		\item[(2.b)]\label{Hyp_Concentration} one has 
		$ \frac{2}{m_{\mathcal{S}}^{2}}\sum_{\gamma \in \mathcal{Z}_{\mathcal{S}}^{*}}\frac{\lvert M(\gamma)\rvert^{2}}{(\beta_{\mathcal{S},0}^{2}+\gamma^{2})} < \frac{1}{2}.
		$
	\end{enumerate}
	Then $\delta(\mathcal{S})$ exists
	and $\left(\delta(\mathcal{S}) - \frac{1}{2}\right)m_{\mathcal{S}} \geq 0$. 
\end{cor}

We have avoided so far the use of the Riemann Hypothesis, weakening the hypotheses made in previous works.
We generalize Rubinstein and Sarnak result \cite[Th. 1.2]{RS} by stating a dichotomy depending on the validity of the Riemann Hypothesis.

\begin{theo}\label{Th_withRH}
Suppose that the set $\mathcal{S}$ satisfies the conditions of Theorem~\ref{Th_DistLim}.
\begin{enumerate}
\item\label{Th1_RH} Suppose the Riemann Hypothesis is satisfied for every $L(f,s)$, $f\in\mathcal{S}$ 
(\textit{i.e.} $\beta_{\mathcal{S},0}=\frac{1}{2}$).
Suppose also that for every $f\in \mathcal{S}$, one has $\re(a_{f})\geq 0$, and that there exists $f\in\mathcal{S}$ such that $\re(a_{f})> 0$.
Then there exists a constant $c$ depending on $\mathcal{S}$ such that
$$\mu_{\mathcal{S}}(\mathbf{R}-[-A,A]) \gg \exp(-\exp(cA)).$$
In particular
$0< \underline{\delta}(\mathcal{S}) \leq \overline{\delta}(\mathcal{S}) <1$.
\item\label{Th2_nRH} Suppose $\beta_{\mathcal{S},0}>\frac{1}{2}$, 
and the sum $\sum_{\gamma \in \mathcal{Z}_{\mathcal{S}}}\frac{1}{\lvert \beta_{\mathcal{S},0} + i\gamma\rvert}$ converges.
Then $\mu_{\mathcal{S}}$ has compact support.
\end{enumerate}
\end{theo}

		This result is proved in section \ref{sub_Support}.
		
\begin{Rk}	
		\begin{enumerate}
			\item The proof of Theorem~\ref{Th_withRH}(\ref{Th1_RH}) is an adaptation to a more general case of the proof given by Rubinstein and Sarnak.
 The hypothesis about $\re(a_{f}) \geq 0$ does not seem very natural, but is necessary in our analysis. 	
		It is satisfied in the case the set $\mathcal{S}$ is a singleton
		i.e. in most of the examples presented in Section~\ref{Sec_Applications}.
		In the case of the prime number race between congruence classes (Theorem~\ref{Th_RS_raceModq}), the condition holds when one studies the race between $1\ [\bmod\ q]$ and another invertible class modulo $q$.
		\item One can also show that the race is inclusive --- i.e. that each contestant leads the race infinitely many times 
		(this is implied by $0< \underline{\delta}(\mathcal{S}) \leq \overline{\delta}(\mathcal{S}) <1$) ---
		assuming GRH and LI and nothing on the $a_{f}$'s (see \cite{RS}).
		In \cite{MartinNg}, Martin and Ng prove that the race is inclusive assuming GRH and a weaker hypothesis than LI based on self-sufficient zeros (see Definition~\ref{Def_selfsuff}).
		\item The hypothesis in Theorem~\ref{Th_withRH}(\ref{Th2_nRH}) is a weak Zero Density hypothesis, 
		we assume that there are not too many zeros off the critical line.
		There are results supporting this hypothesis, see for example \cite[Chap. 10]{IK} for the Riemann zeta function and Dirichlet $L$-functions. 
		More generally, Kaczorowski and Perelli (\cite[Lem. 3]{KP}) proved a stronger version of this hypothesis in the case of a Selberg class $L$-function of degree $d$ with $\beta_{f,0} \geq 1 - \frac{1}{4(d+3)}$.
	\end{enumerate}
	
\end{Rk}

%% file: Applications.tex
In this section we present two kinds of applications.
We first find some results of the literature as special cases of our general result.
Most of them were conditional on GRH, they are now unconditional.
In a second part, we use the fact that analytic $L$-functions describe a wide range of $L$-functions to present new applications of Chebyshev's races.

\subsection{Proofs of old results under weaker assumptions}

\subsubsection{Sign of the second term in the Prime Number Theorem.}

The first example of analytic $L$-function is Riemann's zeta function (see Example~\ref{Ex_Lfunc}(\ref{Ex_Lfunc_zeta})).
Adapting Theorem~\ref{Th_DistLim} to $\zeta$ yields an unconditional proof of the existence of the logarithmic limiting distribution for the race $\pi(x)$ versus $\Li(x)$ (see e.g. \cite{Wintner}, \cite[p. 175]{RS} for previous results under RH).
\begin{theo}\label{Th_race_pi_Li}
With the notations as in Section~\ref{subsec_setting},
the function 
$$E_{\zeta}(x) = \frac{\log x}{x^{\beta_{\zeta,0}}}(\pi(x) -\Li(x))$$
has a limiting logarithmic distribution $\mu_{\zeta}$ on $\mathbf{R}$.
There exists a positive constant $C$ such that one has 
$$\mu_{\zeta}(\mathbf{R}-[-A,A]) \ll \exp(-C\sqrt{A}).$$
Moreover, the expected value of $\mu_{\zeta}$ is 
$m_{\zeta} = -\delta_{\beta_{\zeta,0} = \frac{1}{2}}.$
\end{theo}

\begin{proof}
This follows from the fact that the function with squared local roots associated to $\zeta$ is $\zeta$ itself, 
moreover $\zeta$ has a pole of order $1$ at $s=1$ and does not vanish over $(0,1)$.
\end{proof}

We deduce the already known idea that morally 
(e.g. under the conditions of Corollary~\ref{Cor_mean_bias}), 
assuming RH, the race $\pi(x) - \Li(x)$ should be biased towards negative values. 
Conversely a bias towards negative values in the race $\pi(x) - \Li(x)$  gives evidence that RH should hold.

\subsubsection{Prime number races between congruence classes modulo an integer.}\label{subsubsec_ex_cong}

The results of Rubinstein and Sarnak \cite[Th. 1.1, Th. 1.2]{RS} in the case of a prime number race with only two contestants $a$ and $b$ modulo $q$ is a particular case of Theorem \ref{Th_DistLim}.
Indeed take $$\mathcal{S}= \lbrace L(\chi,\cdot) : \chi \bmod\ q, \chi\neq \chi_{0}\rbrace$$
and $a_{\chi} = \overline{\chi}(a) - \overline{\chi}(b)$.
We obtain an uncondtional (i.e. without assuming GRH) proof of \cite[Th. 1.1, Th. 1.2]{RS}.

\begin{theo}\label{Th_RS_raceModq}
Let $q$ be an integer, $a \neq b\ [\bmod\ q]$ two invertible residue classes.
Define $$\beta_{q,0} = \sup\lbrace \re(\rho) : \exists \chi \bmod\ q, \chi\neq \chi_{0},  L(\chi,\rho)=0\rbrace.$$
The function 
$$E_{q;a,b}(x) = \frac{\log x}{x^{\beta_{q,0}}}(\pi(x,q,a) -\pi(x,q,b))$$
has a limiting logarithmic distribution $\mu_{q;a,b}$ on $\mathbf{R}$.
There exists a positive constant $C$ (depending on $q$) such that one has 
$$\mu_{q;a,b}(\mathbf{R}-[-A,A]) \ll \exp(-C\sqrt{A}).$$
Moreover, suppose GRH is satisfied (i.e. $\beta_{q,0}=1/2$) and $L(\chi,1/2)\neq 0$ for every $\chi \bmod\ q$, $\chi\neq \chi_{0}$.
Then the expected value of $\mu_{q;a,b}$ is 
$$m_{q;a,b} = \sum_{\substack{\chi \bmod\ q\\ \chi^{2} = \chi_{0}}} (\chi(b) - \chi(a)).$$
\end{theo}

\begin{Rk}
In particular under these hypotheses and the conditions of Corollary~\ref{Cor_mean_bias}(a):
\begin{enumerate}
\item if $ab^{-1}$ is a square, then there is no bias, 
\item otherwise, the bias is in the direction of the non quadratic residue.
\end{enumerate}
\end{Rk}

Following the idea of \cite{RS} and \cite{Fiorilli_HighlyBiased}, 
we can also study the prime number race between the subsets of quadratic residues and non-residues modulo an integer $q$. 
For this, take
$$\mathcal{S}= \lbrace L(\chi,\cdot) : \chi \bmod\ q, \chi\neq \chi_{0}, \chi^{2} = \chi_{0}\rbrace,$$
and for each real character $\chi$ modulo $q$, take $a_{\chi}=\frac{1}{\rho(q)} := [(\mathbf{Z}/q\mathbf{Z})^{\times} : (\mathbf{Z}/q\mathbf{Z})^{\times (2)}]^{-1}$.
We apply Theorem~\ref{Th_DistLim} to this setting and get the following result.
\begin{theo}
Let $q\geq 3$ be an integer.
Define $$\beta^{(2)}_{q,0} = \sup\lbrace \re(\rho) : \exists \chi \bmod\ q, \chi\neq \chi_{0}, \chi^{2} = \chi_{0},  L(\chi,\rho)=0\rbrace.$$
The function
$$E_{q;{\rm R},{\rm NR}}(x) = \frac{\log(x)}{\rho(q)x^{\beta^{(2)}_{q,0}}}((\rho(q) -1)\pi(x;q,{\rm R}) -\pi(x;q,{\rm NR}))$$
has a limiting logarithmic distribution $\mu_{q;\rm{R},\rm{NR}}$ on $\mathbf{R}$.

Moreover, suppose GRH is satisfied for all Dirichlet $L$-function of real characters modulo $q$
and that $L(\chi,1/2)\neq 0$ for every $\chi \bmod\ q$, $\chi\neq \chi_{0}$, $\chi^{2}=\chi_{0}$.
Then the average value of $\mu_{q;{\rm R},{\rm NR}}$ is
$$m_{q;a,b} = \frac{1 -\rho(q)}{\rho(q)}.$$
\end{theo}

Thus we have obtained an unconditional proof (without GRH) of \cite[Th. 1.1]{RS} in the case of the prime number race  between the subsets of quadratic residues and non-residues modulo an integer $q$  (see also \cite[Lem. 2.2]{Fiorilli_HighlyBiased}).
Under GRH, the mean of the logarithmic limiting distribution is negative, 
hence morally we should find a bias towards negative values (i.e. towards non quadratic residues).
This result has already been used by Fiorilli in \cite{Fiorilli_HighlyBiased} to find arbitrarily biased races between residues and non-residues modulo integers having a lot of prime factors (so that the mean value is as far from $0$ as possible).

\subsubsection{$L$-functions of automorphic forms on $GL(m)$.}

As announced in Example~\ref{Ex_Lfunc}.(\ref{Ex_Lfunc_cusp}), by results on automorphic $L$-functions (\cite{BumpFriedberg,BumpGinzburg,GodementJacquet,JacquetShalika}), we only miss the Ramanujan--Petersson conjecture to ensure that 
$L$-functions associated to irreducible cuspidal automorphic forms on $GL(m)$ 
are analytic $L$-functions in the sense of Definition~\ref{Def_Lfunc}.
We get a version of Theorem~\ref{Th_DistLim} for automorphic $L$-functions conditional on the Ramanujan--Petersson conjecture.
\begin{theo}\label{Th_Aut_bias}
Let $L(\pi,s)$ be a real $L$-function associated to an irreducible unitary cuspidal automorphic representation of $GL(m)$ 
with $m\geq 2$.
Suppose the Ramanujan--Petersson conjecture holds for $L(\pi,s)$.
Then, following the notations of Section~\ref{subsec_setting}, the function
$E_{\pi}(x) = \frac{\log x}{x^{\beta_{\pi,0}}}\sum_{p\leq x} \lambda_{\pi}(p)$
has a limiting logarithmic distribution $\mu_{\pi}$.

Moreover under GRH for $L(\pi,s)$, the mean of $\mu_{\pi}$ is
$m_{\pi} = \pm 1 -2m(L(\pi,\cdot),1/2)\neq 0$.
\end{theo}

This result should be compared to \cite[Cor. 1.5]{ANS} where GRH is assumed but not the Ramanujan--Petersson Conjecture.
Morally, under GRH, since the mean value is non zero, we expect that the prime number race associated to such an $L$-function has always a bias. 

\begin{proof}
Under GRH for $L(\pi,s)$, we study the behaviour of the function $$L(\pi^{(2)},s) = L(\Sym^{2}\pi,s)L(\wedge^{2}\pi,s)^{-1}$$ around $s=1$.
By \cite[Th. 1.1]{Sha}, in the case $\pi$ is an irreducible non trivial representation,
the functions  $L(\Sym^{2}\pi,s)$, $L(\wedge^{2}\pi,s)$ do not vanish at $s=1$.
Moreover one has $L(\pi\otimes\pi,s) = L(\Sym^{2}\pi,s)L(\wedge^{2}\pi,s)$,
and (\cite[App.]{MW}) this function has a simple pole at $s=1$ when $\pi$ is self-dual (i.e. $L(\pi,s)$ is real).
Hence there are only two possibilities:
\begin{itemize}
\item either $L(\Sym^{2}\pi,s)$ has a simple pole at $s=1$ and $m(L(f^{(2)},\cdot),1) = -1$,
\item or $L(\wedge^{2}\pi,s)$ has a simple pole at $s=1$ and $m(L(f^{(2)},\cdot),1) = 1$.
\end{itemize} 
Theorem~\ref{Th_Aut_bias} follows.
\end{proof}

In the case $m=2$, the Ramanujan--Petersson conjecture has been proved by works of Deligne and Deligne--Serre \cite{Deligne_Weil1,DeligneSerre}.
In particular the normalized Hasse--Weil $L$-function associated to an elliptic curve defined over $\mathbf{Q}$ is an analytic $L$-function (see Example \ref{Ex_Lfunc}(\ref{Ex_Lfunc_Modular2})).
Hence we deduce \cite[Lem. 2.3, Lem. 2.6, Lem. 3.4]{FioEC} from Theorem \ref{Th_DistLim}.

\begin{prop}\label{Prop_Fi_EC}
Let $E/\mathbf{Q}$ be an elliptic curve, and $L(E,s)$ its normalized Hasse $L$-function.
The function 
$$E_{E}(x) = \frac{\log x}{x^{\beta_{E,0}}}\sum_{p\leq x}\frac{a_{p}(E)}{\sqrt{p}}$$
has a limiting logarithmic distribution $\mu_{E}$ on $\mathbf{R}$.
Moreover, suppose GRH is satisfied for $L(E,s)$.
Then the mean of $\mu_{E}$ is 
$$m_{E} = -2r_{an}(E) + 1$$
where $r_{an}(E)$ is the analytic rank of $E$.
\end{prop}

\begin{proof}
In the case of a Hasse--Weil $L$-function attached to an elliptic curve $E/\mathbf{Q}$, 
one has $L(\wedge^{2}, E, s)=\zeta(s)$.
Proposition~\ref{Prop_Fi_EC} follows.
\end{proof}

We observe the two distinct cases pointed out by Mazur:
either $r_{an}(E) = 0$ and we should expect a bias towards positive values,
or $r_{an}(E) >0$ and we should expect a bias towards negative values.
As Fiorilli noticed in \cite{FioEC} we can expect an arbitrarily large bias in the case the rank of the elliptic curve is arbitrarily large compared to the variance of the distribution.

\subsection{New applications}

\subsubsection{Chebyshev's bias and prime numbers of the form $a^{2} + D b^{2}$.}

In \cite{BS}, Beukers and Stienstra give several examples of $L$-functions of degree $2$ related to $K3$ surfaces.
Precisely they define the three following functions:
$$L_{D}(s) = \prod_{p\nmid 2D}\left(1-a_{p}p^{-s} + \left(\frac{-D}{p}\right)p^{-2s}\right)^{-1}$$
for $D= 2$, $3$ and $4$,
where
$$a_{p} =\left\lbrace
\begin{array}{ll}
0 & \mbox{if $\left(\frac{-D}{p}\right)=0$ or $-1$,}\\
2\frac{a^{2} - Db^{2}}{p} & \mbox{if one can write $p= a^{2} + Db^{2}$ with $a,b\in \mathbf{N}$.}
\end{array}
\right.$$
By \cite[Th. 14.2]{BS} (and \cite{Schoeneberg} in the case $D=4$), those $L$-functions are associated to cusp forms of weight $3$ and level $4D$.
In particular they satisfy Definition \ref{Def_Lfunc}.

To these functions $L_{D}$ one can associate the prime number race that consists in understanding the sign of
\begin{align}\label{Eq_sum2squares}
E_{D}(x) = \frac{\log x}{x^{\beta_{0,D}}}\sum_{p=a^{2} + Db^{2}\leq x}2\frac{a^{2} - Db^{2}}{a^{2} + Db^{2}}.
\end{align}
The adaptation of Theorem~\ref{Th_DistLim} to this context is the following result.

\begin{theo}\label{Th_sum2squares}
	For $D= 4$, $2$ and $3$,
	$E_{D}$ has a limiting logarithmic distribution whose average value is 
	$$-\frac{m(L_{D},\beta_{0})}{\beta_{0}} -\delta_{\frac{1}{2},\beta_{0}}\leq 0.$$
\end{theo}

\begin{proof}
By \cite[Th. 14.2]{BS}, we are in the situation of Theorem~\ref{Th_Aut_bias}.
In particular the limiting logarithmic distribution exists. 
One can compute that, for each of the three cases $D=2,3$ and $4$, and for every $p$,
the products of the local roots are
$\alpha_{1}(p)\alpha_{2}(p) = \left(\frac{-D}{p}\right)$.
We deduce  
$$L\left(\wedge^{2}f_{D},s\right)=\prod\left(1-\left(\frac{-D}{p}\right)p^{-s}\right)^{-1},$$ 
and in particular this function is entire.
Hence (by \cite[App.]{MW}) the function $L(\Sym^{2}f_{D},s)$ has a pole of multiplicity $1$ at $s=1$.
In conclusion the function $L(f_{D}^{(2)},s)$ has a pole of multiplicity $1$ at $s=1$. 
\end{proof}

We can interpret this result by saying that in the decomposition $p=a^{2} + Db^{2}$, the term $Db^{2}$ is often larger than $a^2$.
The Figures \ref{fig_CourseD2}, \ref{fig_CourseD3} and \ref{fig_CourseD4} represent respectively the races between $a^{2}$ and $-Db^{2}$ for $D=2,3,4$. We used \texttt{sage} and Cornacchia's algorithm to obtain the values of the functions $S_{D}(x):=\sum\limits_{p=a^{2} + Db^{2}\leq x}\frac{a^{2} - Db^{2}}{a^{2} + Db^{2}}$ for $x$ between $0$ and $2.10^{7}$.
We see on these figures that it is natural to expect a bias towards the negative values.

\begin{figure}
   \begin{minipage}[c]{.46\linewidth}
      \includegraphics[scale=0.16]{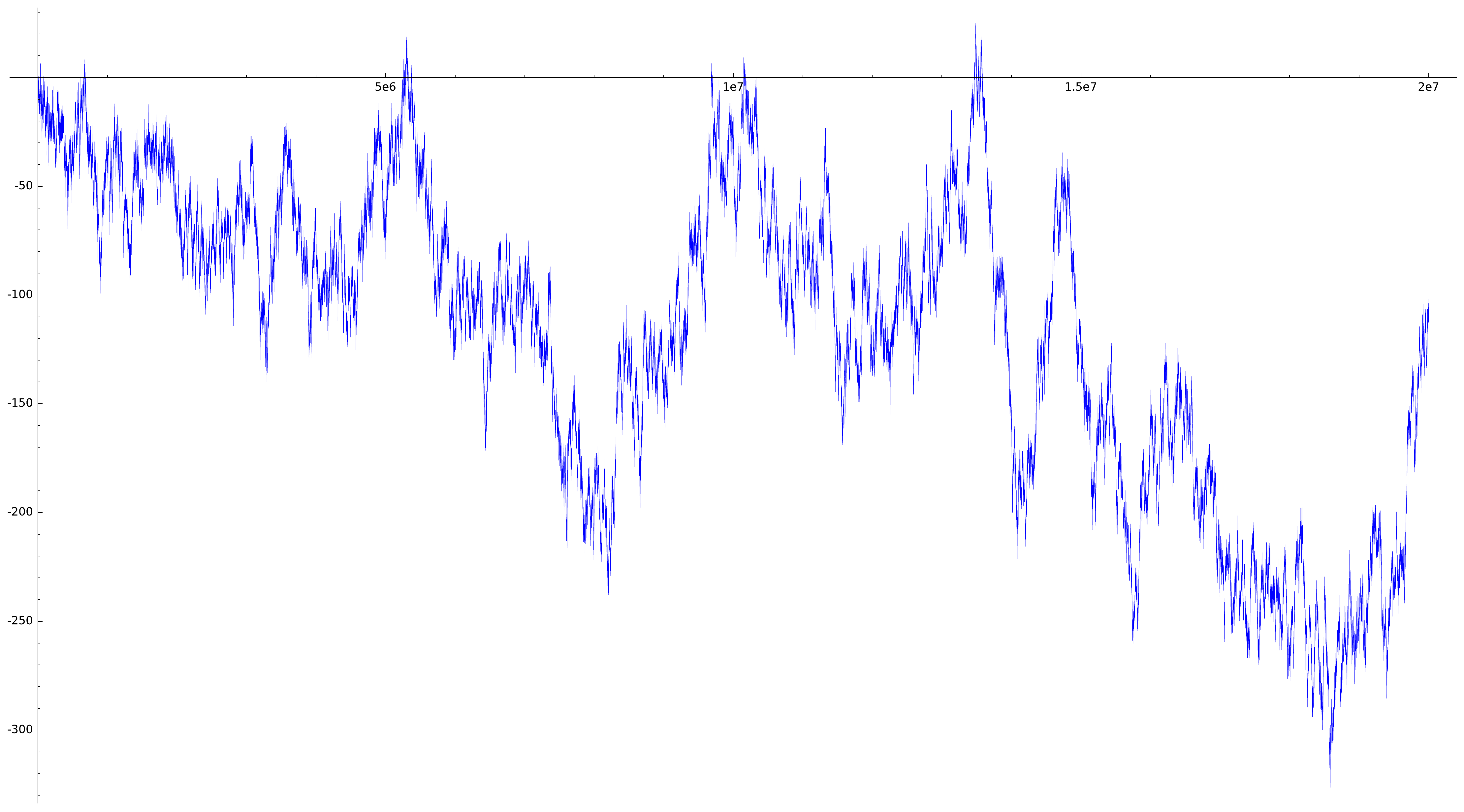}
      
      \caption{Values of $S_{2}(x)$ in the range $(0,2.10^{7})$}
      \label{fig_CourseD2}
   \end{minipage} \hfill
   \begin{minipage}[c]{.46\linewidth}
      \includegraphics[scale=0.16]{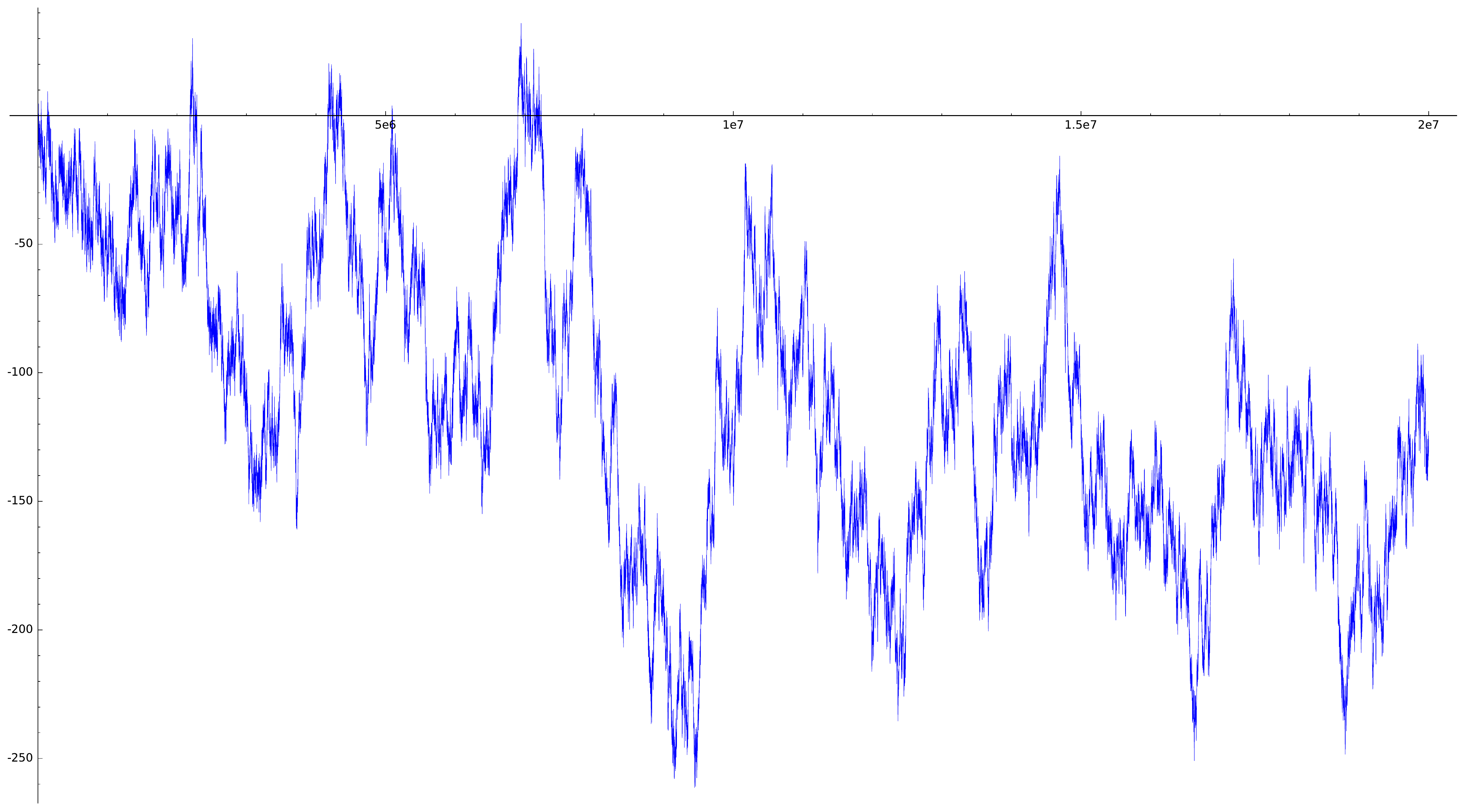}
      
      	\caption{Values of $S_{3}(x)$ in the range $(0,2.10^{7})$}
      	\label{fig_CourseD3}
   \end{minipage}
   \begin{minipage}[c]{.46\linewidth}
\includegraphics[scale=0.16]{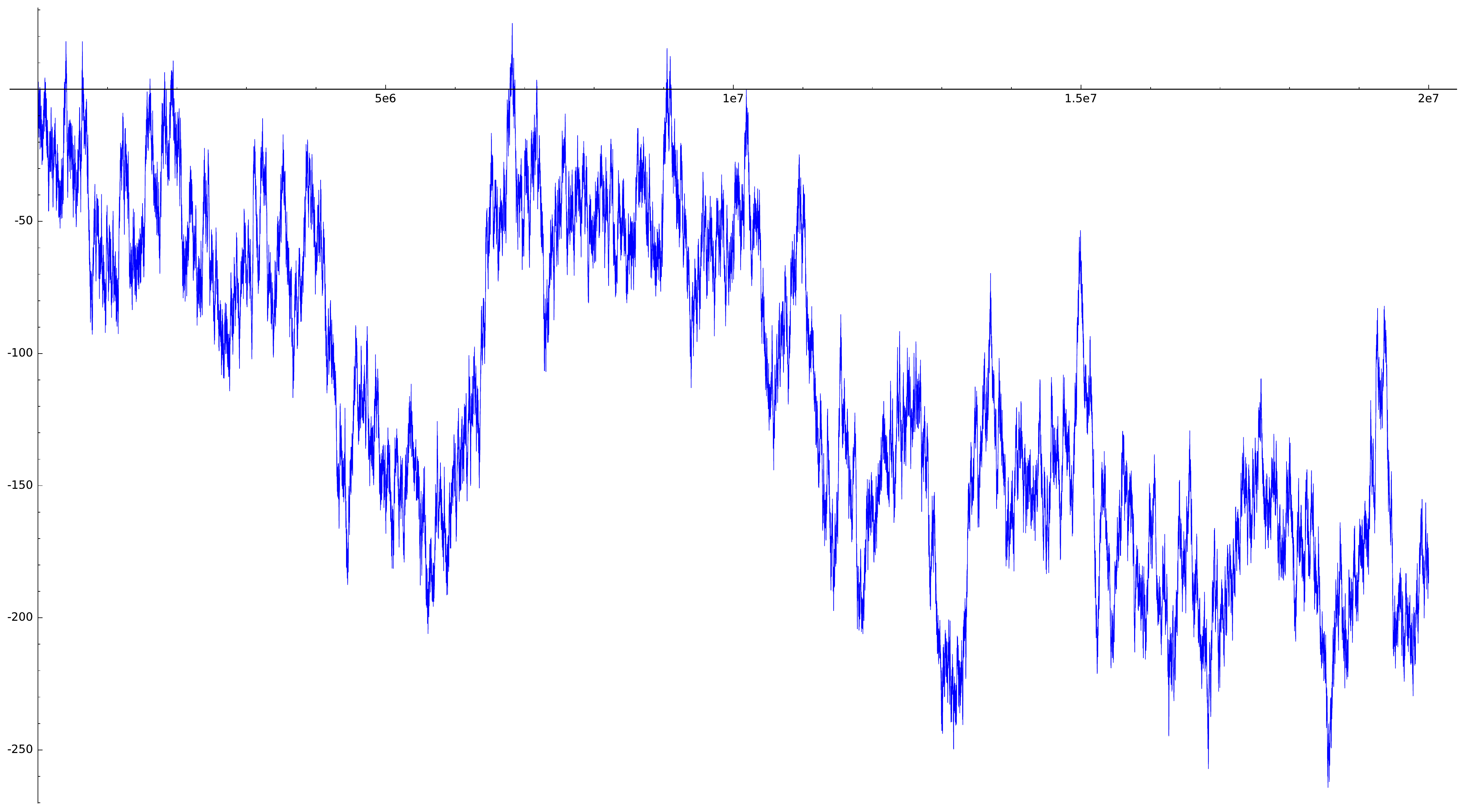}

\caption{Values of $S_{4}(x)$ in the range $(0,2.10^{7})$}
\label{fig_CourseD4}
\end{minipage} \hfill
   \begin{minipage}[c]{.46\linewidth}
	\includegraphics[scale=0.16]{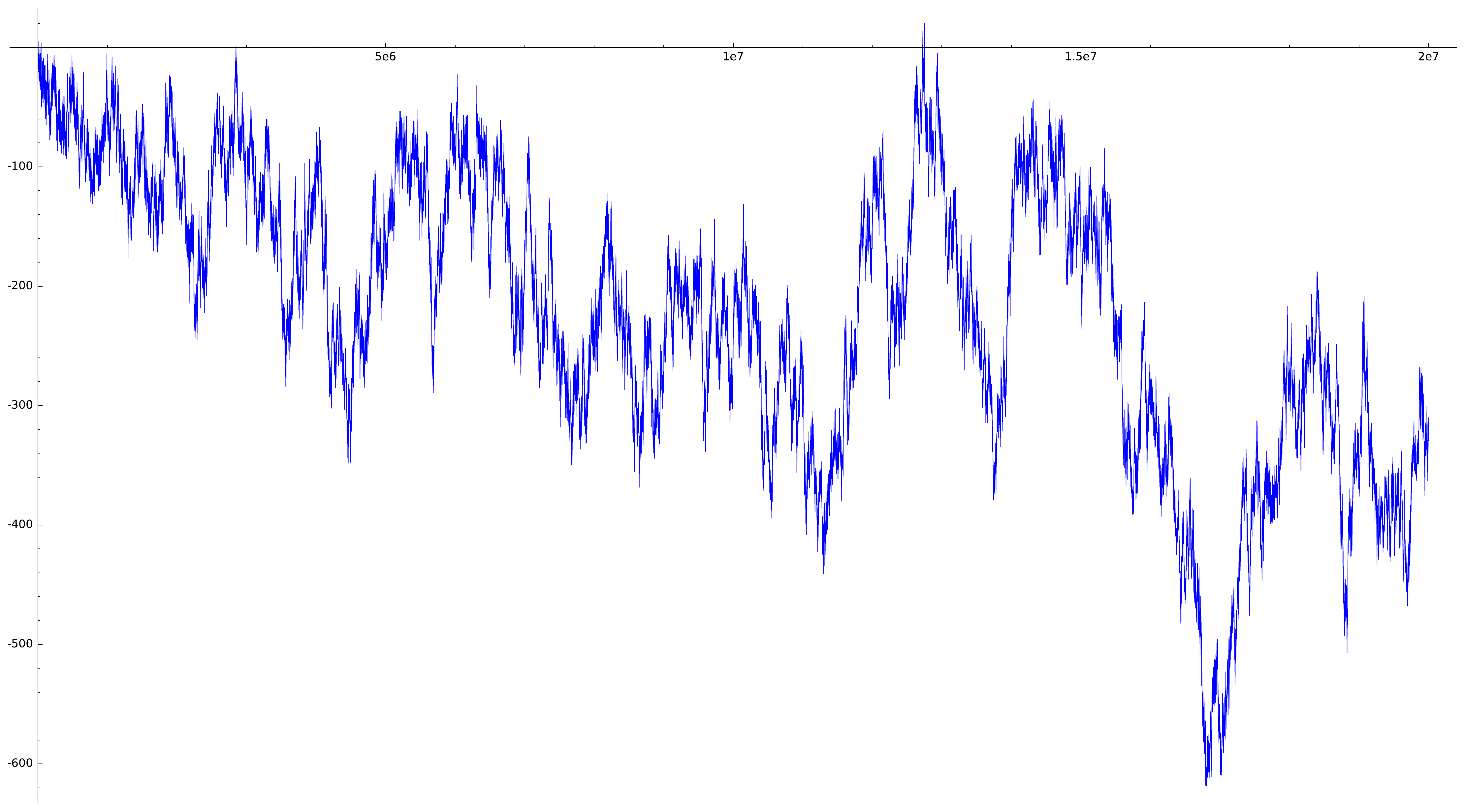}

\caption{Values of $S_{\lambda}(x)$ in the range $(0,2.10^{7})$}
\label{fig_CourseGauss}
\end{minipage} \hfill

\end{figure}

\subsubsection{Prime number races for angles of Gaussian primes}

As Z. Rudnick pointed out to us, in the case $D=4$, the prime number race \eqref{Eq_sum2squares} is related to the question of the bias in the distribution of the angles of the Gaussian primes.
Let $\lambda$ be the Hecke character on $\mathbf{Z}[i]$ defined by $\lambda(z) = \left(\frac{z}{\bar{z}}\right)^{2}$.
The $L$-function $L(\lambda,s) = \prod_{p} \prod_{\mathfrak{p}\mid p} (1 - \lambda(\mathfrak{p})N\mathfrak{p}^{-s})^{-1}$ (seen as a Euler product over rational primes)
is an analytic $L$-function in the sense of Definition~\ref{Def_Lfunc}.
In the case $p\equiv 1\ [\bmod\ 4]$ the local factor is $(1 - \cos(4\theta_{p})p^{-s} + p^{-2s})^{-1}$
where $\pm\theta_{p}$ are the angle of the Gaussian primes dividing $p$ (they are defined modulo $\frac{\pi}{2}$).
The prime number race associated with this situation consists in understanding the sign of the function
\begin{align*}
E_{\lambda}(x) = \frac{\log x}{x^{\beta_{0,\lambda}}}\sum_{\substack{p\leq x\\ p\equiv 1\ [\bmod\ 4]}} \cos(4\theta_{p}).
\end{align*}

\begin{theo}\label{Th_BiasGauss}
	The function $E_{\lambda}$ admits a limiting logarithmic distribution whose average value is 
negative.
\end{theo}

\begin{proof}
	See for example \cite[Th. 7-19]{RamakrishnanValenza} to verify the hypotheses of Definition~\ref{Def_Lfunc}. 
	In the case $p\equiv 3\ [\bmod\ 4]$ the local factor of $L(\lambda,s)$ is $(1 - p^{-2s})^{-1}$, so the local roots as a $L$-function of degree $2$ over $\mathbf{Q}$ are $\pm1$.
	Hence for every $p$,
	the product of the local roots is
	$\left(\frac{-1}{p}\right)$.
	Thus  
	$L\left(\wedge^{2}\lambda,s\right)=\prod\left(1-\left(\frac{-1}{p}\right)p^{-s}\right)^{-1}$
	is entire.
	Similarly the function $L(\Sym^{2}\lambda,s)=\zeta(s)L(\lambda^2,s)$ has a pole of multiplicity $1$ at $s=1$.
	In conclusion the function $L(\lambda^{(2)},s)$ has a pole of multiplicity $1$ at $s=1$.
\end{proof}

This result implies that the corresponding prime number race should be biased towards negative values.
This can be guessed from computations. In Figure~\ref{fig_CourseGauss} we used \texttt{sage} to compute the values of the function $S_{\lambda}(x):=\sum\limits_{\substack{p\leq x\\ p\equiv 1\ [\bmod\ 4]}} \cos(4\theta_{p})$ for $x$ between $0$ and $2.10^{7}$.
The consequence of such a bias towards negative values is that we expect that the Gaussian primes are more often closer to the line $y=x$ than to the axes.

\subsubsection{Correlations for two elliptic curves.}\label{Ex_2ellcurves}

As advertised in Example~\ref{Ex_Lfunc}(\ref{Ex_Lfunc_RankinSelberg}),
we can study the prime number race associated to a Rankin--Selberg product of $L$-functions.
Let $E_{1}$ and $E_{2}$ be two non-isogenous non-CM elliptic curves defined over $\mathbf{Q}$.
By works of Wiles, Taylor--Wiles, Breuil--Conrad--Diamond--Taylor (\cite{Wiles_ModEll,TaylorWiles_Modularity,BCDT_Modularity}),
there exists cuspidal modular forms $f_{1}\neq f_{2}$ associated to $E_{1}$ and $E_{2}$ respectively
(i.e. the corresponding normalized $L$-functions are the same). 
One has 
$$\lambda_{f_{i}}(p) = \frac{a(E_{i},p)}{\sqrt{p}} = \frac{p+1 - \lvert E_{i}(\mathbf{F}_{p})\rvert}{\sqrt{p}}.$$

In the case $E_{1}$ and $E_{2}$  do not become isogenous over a quadratic extension of $\mathbf{Q}$, by \cite{Ramakrishnan}
the Rankin--Selberg convolution $L(f_{E_{1}}\otimes f_{E_{2}},\cdot)$
is a real analytic $L$-function in the sense of Definition \ref{Def_Lfunc}.
Its coefficients are $\lambda(p) = a_{p}(E_{1})a_{p}(E_{2})/p$.
Moreover, if we assume that the curves $E_{1}$ and $E_{2}$  do not become isogenous over any abelian extension of $\mathbf{Q}$,
a strong version of Conjecture \ref{Sato-Tate} holds for these coefficients (see \cite[Th. 5.4]{HarrisAutEC}). 

Hence we can apply Theorem \ref{Th_DistLim}.
The function
$E(x)= \frac{\log x}{x^{\beta_{0}}} \sum_{p\leq x}\frac{a_{p}(E_{1})a_{p}(E_{2})}{p}$
admits a limiting logarithmic distribution,
and we can give its mean explicitly.
The term
$m(L(f_{E_{1}}\otimes f_{E_{2}},\cdot),1/2)$ may not be easy to evaluate,
but 
$m(L((f_{E_{1}}\otimes f_{E_{2}})^{(2)},\cdot),1)$ can be computed.
From these considerations we obtain the following result.

\begin{theo}\label{Prop_CorrEllcurves}
Let $E_{1}$ and $E_{2}$ be two non-CM elliptic curves defined over $\mathbf{Q}$.
Assume $E_{1}$ and $E_{2}$  do not become isogenous over a quadratic extension of $\mathbf{Q}$.
The function $$E(x)= \frac{\log x}{x^{\beta_{0}}} \sum_{p\leq x}\frac{a_{p}(E_{1})a_{p}(E_{2})}{p}$$
admits a limiting logarithmic distribution.
Assume the Riemann Hypothesis holds then the logarithmic distribution has negative mean value.
\end{theo}

This result is a consequence of the following lemma.

\begin{lem}\label{lm_productCE}
Let $E_{1}$ and $E_{2}$ be two non-CM elliptic curves defined over $\mathbf{Q}$.
Suppose that $E_{1}$ and $E_{2}$ do not become isogenous over any quadratic extension of $\mathbf{Q}$,
then
$$m(L((f_{E_{1}}\otimes f_{E_{2}})^{(2)},\cdot),1)=-1$$
\end{lem}
\begin{proof}
To fix the notation we write for  $i=1,2$,
\begin{align*}
L(f_{E_{i}},s) = \prod_{p}(1-\pi_{i}p^{-s})^{-1}(1-\overline{\pi_{i}}p^{-s})^{-1}.
\end{align*}
The local roots of $L(f_{E_{1}}\otimes f_{E_{2}},\cdot)$ at $p$ are $\pi_{1}\pi_{2}$, $\overline{\pi_{1}}\pi_{2}$, $\pi_{1}\overline{\pi_{2}}$ and $\overline{\pi_{1}\pi_{2}}$.
Hence
\begin{align*}
L(\wedge^{2}(f_{E_{1}}\otimes f_{E_{2}}),s) =
\prod_{p}\prod_{i=1,2}(1-\pi_{i}^{2}p^{-s})^{-1}(1-\overline{\pi_{i}}^{2}p^{-s})^{-1}(1-p^{-s})^{-1} 
= L(\Sym^{2}f_{E_{1}},s)L(\Sym^{2}f_{E_{2}},s),
\end{align*}
and
\begin{align*}
L(\Sym^{2}(f_{E_{1}}\otimes f_{E_{2}}),s) 
&= L(\Sym^{2}f_{E_{1}}\otimes\Sym^{2}f_{E_{2}},s)\zeta(s).
\end{align*}
For $i = 1,2$, the function $L(\Sym^{2}f_{E_{i}},\cdot)$ is holomorphic and does not vanish at $s=1$.
By \cite{Ramakrishnan} one can associate to $f_{E_{1}}\otimes f_{E_{2}}$ a cuspidal irreducible representation of $GL(4)$ with the same $L$-function, 
hence by \cite[App.]{MW} the function 
\begin{align*}
L((f_{E_{1}}\otimes f_{E_{2}})\otimes(f_{E_{1}}\otimes f_{E_{2}}) ,s) = L(\Sym^{2}(f_{E_{1}}\otimes f_{E_{2}}),s)L(\wedge^{2}(f_{E_{1}}\otimes f_{E_{2}}),s)
\end{align*}
has a pole of multiplicity $1$ at $s=1$.
As a consequence:
\begin{align*}
L((f_{E_{1}}\otimes f_{E_{2}})^{(2)},s) = L(\Sym^{2}(f_{E_{1}}\otimes f_{E_{2}}),s)L(\wedge^{2}(f_{E_{1}}\otimes f_{E_{2}}),s)^{-1}
\end{align*}
has a pole of multiplicity $1$ at $s=1$.
\end{proof}

The proof of Theorem \ref{Prop_CorrEllcurves}
then follows from Theorem \ref{Th_DistLim} and Lemma \ref{lm_productCE}.
Under the Riemann Hypothesis, 
the average value of the limiting logarithmic distribution is
$$-2m\left(L(f_{E_{1}}\otimes f_{E_{2}},\cdot),\frac{1}{2}\right) - 1 <0.$$

We may interpret this result by saying that given two non-isogenous elliptic curves (in the strong sense used above),
the coefficients $a_{p}(E_{1})$ and $a_{p}(E_{2})$ often have opposite signs.
The Figures \ref{fig_CourseE1E2}, \ref{fig_CourseE0E2}, \ref{fig_CourseE1E0} and \ref{fig_CourseE0E'0} represent various prime number races for the correlations of the signs of the $a_{p}(E)$ for two elliptic curves. 
We used four elliptic curves that we can define by an affine model as follows:
$$\begin{array}{cc}
E_{1} : y^{2} + y = x^{3} - x , &  E_{2} : y^{2} + y = x^{3} + x^{2} - 2x ,\\
E_{0} : y^{2} + y = x^{3} - x^{2} , &  E'_{0} : y^{2} + y = x^{3} + x^{2} + x.
\end{array}$$
The elliptic curves have algebraic rank respectively equal to $1$, $2$ and $0$.
We used \texttt{sage} and the counting points algorithm for elliptic curves implemented in \texttt{pari} to obtain the values of the functions $S_{E_{i},E_{j}}(x):=\sum_{p\leq x}\frac{a_{p}(E_{i})a_{p}(E_{j})}{p}$ for $x$ between $0$ and $5.10^{6}$.
The bias towards negative values can be guessed from Figures \ref{fig_CourseE1E2} and \ref{fig_CourseE0E2}, 
it is less clear on Figures \ref{fig_CourseE1E0} and \ref{fig_CourseE0E'0}.
The bias may be smaller in the last two cases and appear only on a larger scale. 

\begin{figure}
   \begin{minipage}[c]{.46\linewidth}
      \includegraphics[scale=0.16]{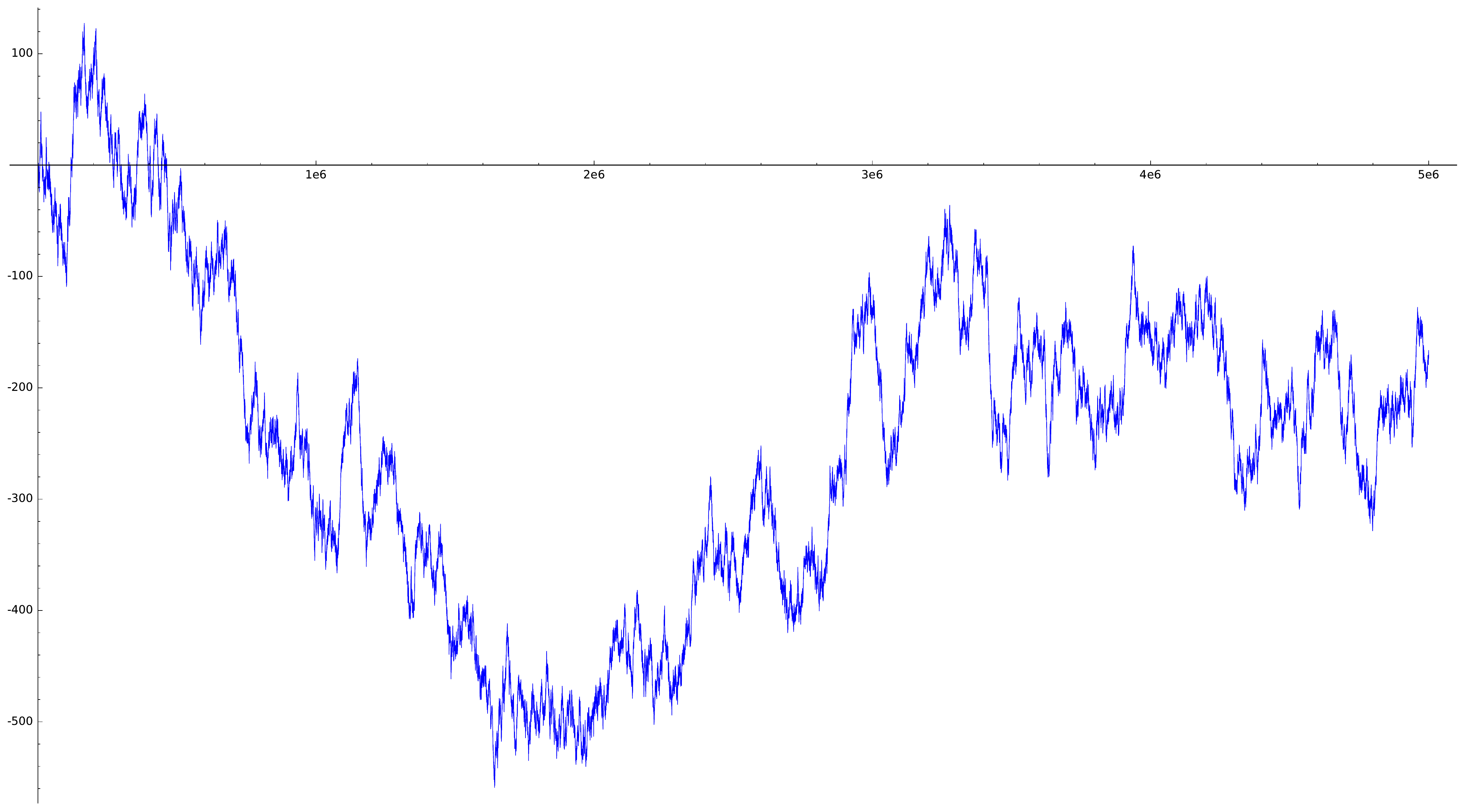}
      
      \caption{Values of $S_{E_{1},E_{2}}(x)$ in the range $(0,5.10^{6})$}
      \label{fig_CourseE1E2}
   \end{minipage} \hfill
   \begin{minipage}[c]{.46\linewidth}
      \includegraphics[scale=0.16]{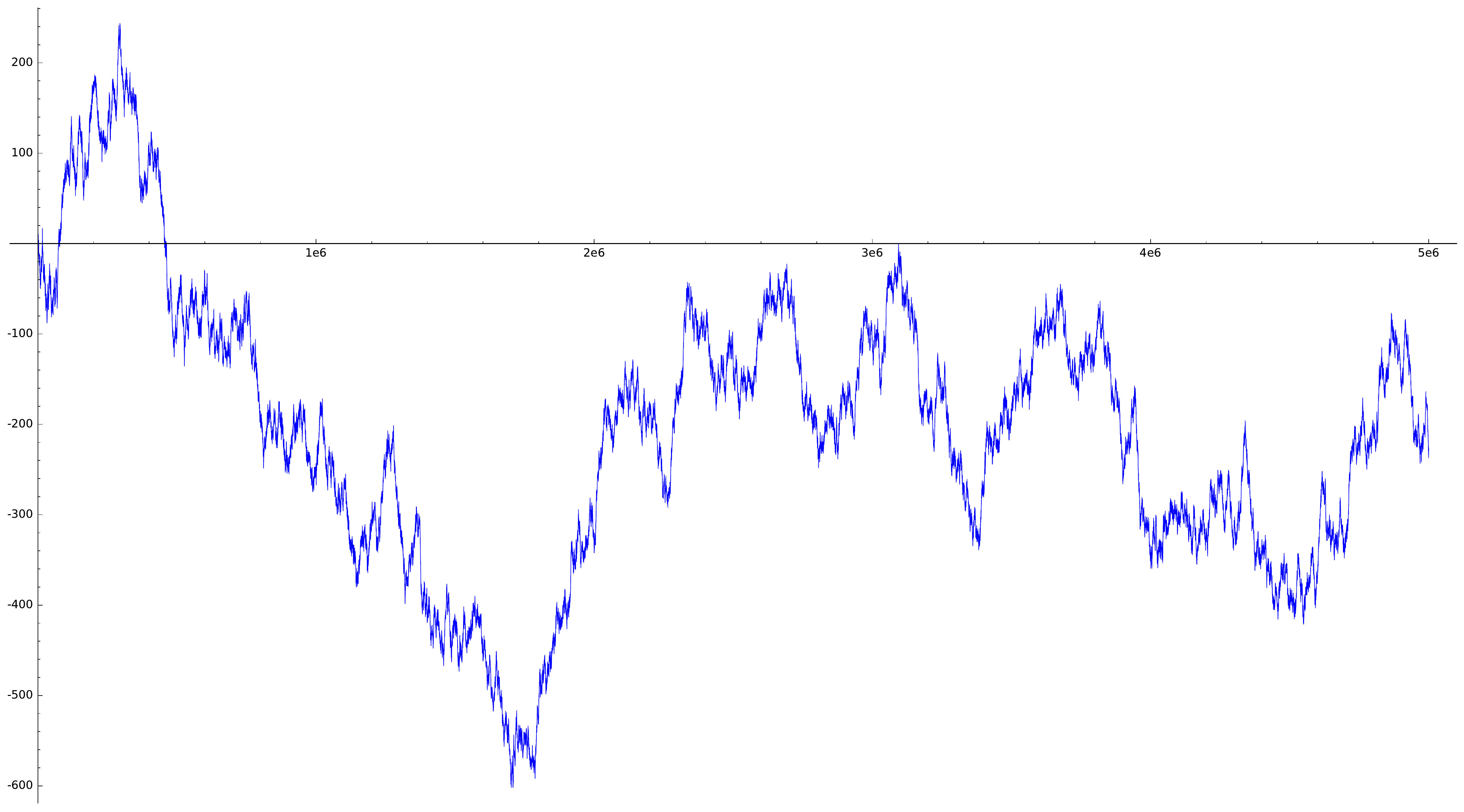}
      
      	\caption{Values of $S_{E_{0},E_{2}}(x)$ in the range $(0,5.10^{6})$}
      	\label{fig_CourseE0E2}
   \end{minipage}

   \begin{minipage}[c]{.46\linewidth}
      \includegraphics[scale=0.16]{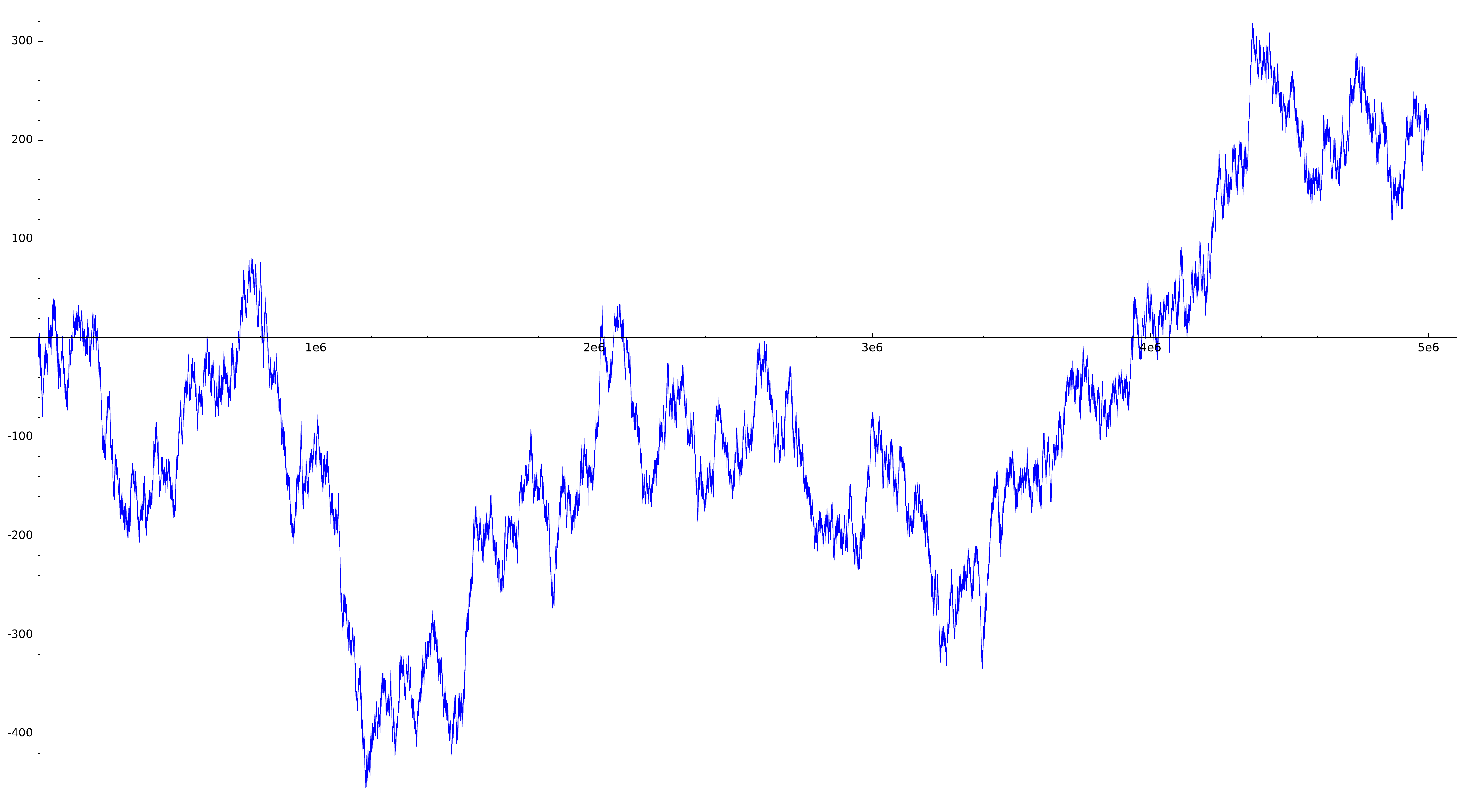}
      
      \caption{Values of $S_{E_{0},E_{1}}(x)$ in the range $(0,5.10^{6})$}
      \label{fig_CourseE1E0}
   \end{minipage} \hfill
   \begin{minipage}[c]{.46\linewidth}
      \includegraphics[scale=0.16]{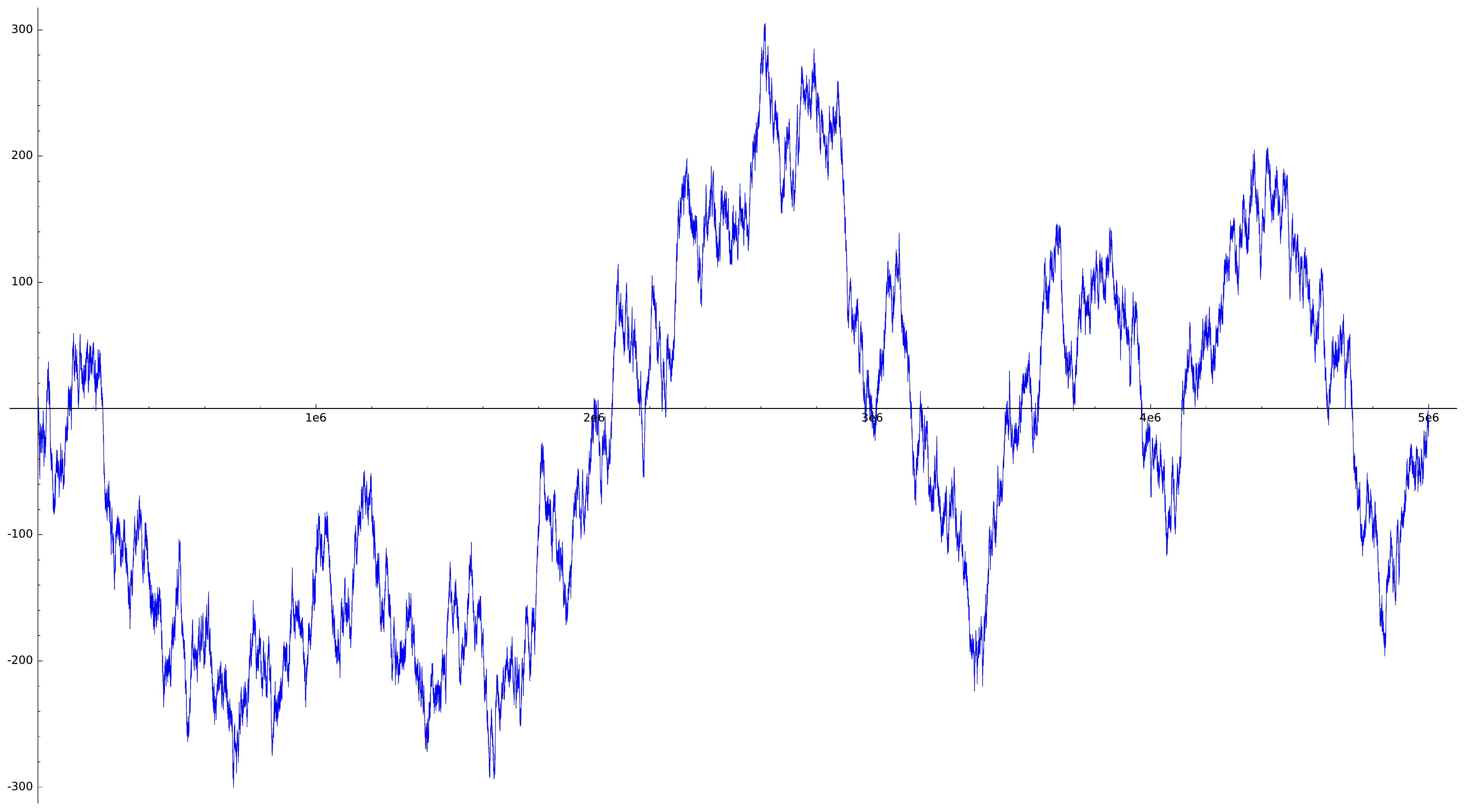}
      
      	\caption{Values of $S_{E_{0},E'_{0}}(x)$ in the range $(0,5.10^{6})$}
      	\label{fig_CourseE0E'0}
   \end{minipage}
\end{figure}

\subsubsection{Jacobian of modular curves.}

Our last example is the prime number race for the $L$-functions of the modular curves.
Let $q$ be a prime number.
We study the prime number race for the sum of the coefficients of all $L$-functions of primitive weight two cusp forms of level $q$.
The $L$-function associated to this race is the finite product
\begin{align*}
\prod_{f\in S_{2}(q)^{*}}L(f,s) = L(J_{0}(q),s),
\end{align*} 
where $J_{0}(q)$ is the Jacobian of the modular curve $X_{0}(q)$ (this factorisation is due to Shimura \cite{Shimura}).
The function $L(J_{0}(q),\cdot)$ is an analytic $L$-function in the sense of Definition \ref{Def_Lfunc}
since it is a product of analytic $L$-functions.

Assuming the Riemann Hypothesis for $L(J_{0}(q),s)$, Theorem \ref{Th_DistLim} applies to the function
$$E_{J_{0}(q)}(x) = \frac{\log x}{\sqrt{x}}\sum_{p\leq x}\sum_{f\in S_{2}(q)^{*}} \lambda_{f}(p).$$
One can conjecture a value for the mean of the limiting logarithmic distribution.
\begin{conj}\label{Conj_largeRank}
One has:
$$m\left(L(J_{0}(q),\cdot),\frac{1}{2}\right) \sim \frac{1}{2}\lvert S_{2}(q)^{*} \rvert$$
as $q\rightarrow \infty$.
\end{conj}
In the articles \cite{KM_upperbound} and \cite{KM_lowerbound}, 
Kowalski and Michel showed that there exist two explicit constants $c<\frac{1}{2}<C$ such that
\begin{align*}
c\lvert S_{2}(q)^{*} \rvert\leq  m(L(J_{0}(q),\cdot),\frac{1}{2}) \leq C\lvert S_{2}(q)^{*} \rvert,
\end{align*}
for all sufficiently large $q$.

The large multiplicity given by Conjecture \ref{Conj_largeRank} may lead us to think that we could get a large bias, but considering all the primitive weight two forms of level $q$ at once, the biases towards positive or negative values should in fact cancel each other.
Precisely:
\begin{theo}\label{Prop_JacModCurve}
Assume the Riemann Hypothesis for $L(J_{0}(q),\cdot)$ (for all $q$)
and assume Conjecture~\ref{Conj_largeRank} holds.
Then the function 
$$E_{J_{0}(q)}(x) = \frac{\log x}{\sqrt{x}}\sum_{p\leq x}\sum_{f\in S_{2}(q)^{*}} \lambda_{f}(p)$$
admits a limiting logarithmic distribution
with mean $o_{q\rightarrow\infty}(\lvert S_{2}(q)^{*} \rvert)$
and variance $\gg \lvert S_{2}(q)^{*} \rvert \log q$.
\end{theo}

\begin{Rk}
In this situation Chebyshev's inequality is not conclusive (see Section \ref{sub_ChebInequality}).
As it is the case of the original work of \cite{RS} the bias probably dissipates as $q\rightarrow\infty$.
If we want to show this, we need a better error term in Conjecture \ref{Conj_largeRank}: we need that $\frac{m_{q}}{\sqrt{\Var_{q}}}\rightarrow 0$ as $q\rightarrow\infty$.
\end{Rk}

For the proof of Theorem~\ref{Prop_JacModCurve}, we compute $m(L(J_{0}(q)^{2},\cdot),1)$. 
\begin{lem}\label{lm_L2Jac}
Let $q$ be an integer.
One has $m(L(J_{0}(q)^{(2)},\cdot),1)= \lvert S_{2}(q)^{*} \rvert $.
\end{lem}
For the record, one has $\lvert S_{2}(q)^{*} \rvert  \sim \frac{q}{12}$.

\begin{proof}
As in the proof of Lemma \ref{lm_productCE}, we use local roots to determine the multiplicities of the zero at $s=1$ of
$L(\wedge^{2}(J_{0}(q)),\cdot)$ and $L(\Sym^{2}(J_{0}(q)),\cdot)$.
For $f\in S_{2}(q)^{*}$, denote by $\alpha_{f}(p)$ and $\overline{\alpha_{f}(p)}$ its local roots.
They satisfy $\alpha_{f}(p)\overline{\alpha_{f}(p)}=1$ if $p\nmid q$.
One has
\begin{align*}
L(\wedge^{2}(J_{0}(q)),s) = \zeta_{q}(s)^{\lvert S_{2}(q)^{*} \rvert}\prod_{f\neq f'}L(f\otimes f',s)
\end{align*}
and
\begin{align*}
L(\Sym^{2}(J_{0}(q)),s) = \prod_{f}L(\Sym^{2}f,s)\prod_{f\neq f'}L(f\otimes f',s).
\end{align*}
Hence $L(\wedge^{2}(J_{0}(q)),\cdot)$ has a pole of multiplicity $\lvert S_{2}(q)^{*} \rvert$ at $s=1$,
and $L(\Sym^{2}(J_{0}(q)),\cdot)$ is holomorphic and does not vanish at $s=1$.
We conclude that $L(J_{0}(q)^{(2)},\cdot)$ has a zero of multiplicity $\lvert S_{2}(q)^{*} \rvert$ at $s=1$.
\end{proof}

\begin{proof}[Proof of Theorem \ref{Prop_JacModCurve}]
It follows from Theorem \ref{Th_DistLim} and Lemma \ref{lm_L2Jac}, under the Riemann Hypothesis that
the mean of the limiting logarithmic distribution is
\begin{align*}
2m\left(L(J_{0}(q),\cdot),\frac{1}{2}\right) - \lvert S_{2}(q)^{*} \rvert.
\end{align*}
If we assume Conjecture \ref{Conj_largeRank} is satisfied, then the mean is $=o(\lvert S_{2}(q)^{*} \rvert)$.
The variance is
\begin{align*}
\sums_{\substack{L(J_{0}(q),\frac{1}{2} +i\gamma)=0 \\ \gamma\neq 0}}\frac{m(L(J_{0}(q),\cdot),\frac{1}{2} +i\gamma)^{2}}{(\frac{1}{4}+\gamma^{2})}
\gg \sum_{\substack{L(J_{0}(q),\frac{1}{2} +i\gamma)=0 \\ \gamma\neq 0}}\frac{1}{(\frac{1}{4}+\gamma^{2})} 
\gg \log(\mathfrak{q}(J_{0}(q)) \asymp \lvert S_{2}(q)^{*} \rvert \log q.
\end{align*}
\end{proof}

%% file: Proof_existence.tex
In this section we prove Theorem~\ref{Th_DistLim} as a consequence of the following result relating $\mu_{\mathcal{S}}$ with the zeros of the $L$-functions.
\begin{prop}\label{Prop_LimOfDist}
Let $\lbrace L(f,\cdot) : f\in\mathcal{S} \rbrace$ be a finite set of analytic $L$-functions such that $\overline{\mathcal{S}}=\mathcal{S}$, 
and $(a_{f})_{f\in \mathcal{S}}$ a set of complex numbers satisfying $a_{\overline{f}}= \overline{a_{f}}$.
Let $T>2$ and 
$$G_{\mathcal{S},T}(x) = m_{\mathcal{S}} -\sum_{\gamma\in\mathcal{Z}_{\mathcal{S}}^{*}(T)}2\re\left(M(\gamma)\frac{x^{i\gamma}}{\beta_{\mathcal{S},0} + i\gamma}\right)$$
where as in Theorem~\ref{Th_DistLim}, for $\gamma$ in $\mathcal{Z}_{\mathcal{S}}^{*}$, one has
$M(\gamma) = \sum_{f\in\mathcal{S}}a_{f}m(L(f,\cdot),\beta_{\mathcal{S},0}+i\gamma)$.
	
The function $G_{\mathcal{S},T}(x)$ admits a limiting logarithmic distribution $\mu_{\mathcal{S},T}$.
Moreover for any bounded Lipschitz continuous function $g$, one has
\begin{align*}
\lim_{T\rightarrow\infty}\int_{\mathbf{R}}g(t)\diff\mu_{\mathcal{S},T}(t) = 
\int_{\mathbf{R}}g(t)\diff\mu_{\mathcal{S}}(t).
\end{align*}
\end{prop}

\begin{Rk}\label{Rk_onProp_LimOfDist}
\begin{enumerate}
\item\label{It_NoZero} In the case $\mathcal{Z}_{\mathcal{S}}$ is empty (it may happen if the Riemann Hypothesis is not satisfied), the functions $G_{\mathcal{S},T}(x)$ are constant, and do not depend of $T$. 
Hence the limiting logarithmic distributions $\mu_{\mathcal{S},T}$ and $\mu_{\mathcal{S}}$ exist and are equal to the Dirac delta function $\delta_{m_{\mathcal{S}}}$.
In particular in the case $\beta_{\mathcal{S},0}=1$, the set $\mathcal{Z}_{\mathcal{S}}$ is empty, and the limiting logarithmic distribution is $\delta_{0}$. 
Hence the only information we get from Theorem \ref{Th_DistLim} is that $S(x) = o\left(\frac{x}{\log x}\right)$.
\item Another approach for this result can be found in \cite{ANS}.
The function $G_{\mathcal{S},T}(e^{y})$ is a trigonometric polynomial, 
and as $T\rightarrow\infty$ it approximates the function $E_{\mathcal{S}}(e^{y})$.
The improvement in our result is that we do not need to assume that the Generalized Riemann Hypothesis holds.
\end{enumerate}
\end{Rk}

To obtain this result (except for the statement about $\Var(X_{\mathcal{S}})$) 
it is enough to consider the case where $\mathcal{S}$ is a singleton $\lbrace f\rbrace$ and $a_{f}=1$ (by linearity).
The proof follows ideas from \cite[Lem. 3.4]{FioEC} and \cite{ANS},
hence we only give the necessary extra details.
The proof is decomposed in the following way.
Subsections~\ref{Subsec_Approx1} and \ref{Subsec_ApproxFin} are dedicated to the proof that 
the functions $G_{f,T}(x)$ are a good approximation for  $E_{f}(x)$.
The existence part of the proposition is proved in subsection \ref{sub_ExistenceLimDist} as a consequence of the Kronecker--Weyl Theorem (of which we sketch the proof in subsection~\ref{Subsec_KW}) and Helly's selection Theorem.
We conclude the proof of Theorem~\ref{Th_DistLim} in subsection~\ref{Subsec_MeanVar} by computing the mean and variance of the limiting logarithmic distribution $\mu_{\mathcal{S}}$.

\subsection{Preliminary result on Kronecker--Weyl Theorem}\label{Subsec_KW}

We prove a generalization of Kronecker--Weyl equidistribution theorem without assuming linear independence following the idea given by Humphries in \cite{MOF_KW}. 

\begin{theo}\label{Th_KW}
	Let $\gamma = (\gamma_1,\ldots,\gamma_N) \in \mathbf{R}^{N}$ be an $N$-uple of arbitrary real numbers.
	Denote $A(\gamma)$ the topological closure of the $1$-parameter group $\lbrace y(\gamma_{1},\ldots,\gamma_{N}) : y\in\mathbf{R}\rbrace/\mathbf{Z}^{N}$ in the $N$-dimensional torus $\mathbf{T}^{N}:= (\mathbf{R}/\mathbf{Z})^{N}$.
	Let $h: \mathbf{T}^{N}\rightarrow \mathbf{C}$ be a continuous function.
	Then $A(\gamma)$ is a sub-torus of $\mathbf{T}^{N}$ and we have
	\begin{equation}\label{Form_KW}
	\lim_{Y\rightarrow\infty}\frac{1}{Y}\int_{0}^{Y}h(y\gamma_{1},\ldots,y\gamma_{N})\diff y 
	= \int_{A(\gamma)}h(a)\diff\omega_{A(\gamma)}(a)
	\end{equation}
	where $\omega_{A}$ is the normalized Haar measure on $A$. 
\end{theo}

This result is a consequence of Fourier analysis in the locally compact Abelian group $\mathbf{T}^{N}$ (see \cite{Folland} and \cite{Rudin_Fourier}).
First, we state in the following result the existence of the Haar measure used in the theorem.

\begin{lem}\label{Lem_closureTorus}
	For $\gamma = (\gamma_{1}, \ldots, \gamma_{N}) \in \mathbf{R}^{N}$, denote $A(\gamma)$ the topological closure of the $1$-parameter group $\lbrace y(\gamma_{1},\ldots,\gamma_{N}) : y\in\mathbf{R}\rbrace/\mathbf{Z}^{N}$ in the $N$-dimensional torus $\mathbf{T}^{N}:= (\mathbf{R}/\mathbf{Z})^{N}$.
	This is a locally compact Abelian group,
	in particular it admits a Haar measure $\omega_{A(\gamma)}$. 
\end{lem}

\begin{proof}
	This follows from the fact that in a topological group, the topological closure of a subgroup is also a subgroup \cite[2.1(c)]{Folland} by continuity of the group operations.
	Then as a closed subspace of a locally compact Hausdorff space is locally compact and Hausdorff when it is given the subspace topology, we deduce that $A(\gamma)$ is a locally compact Abelian group.
	For the existence of the Haar measure on a locally compact Abelian group (unique up to multiplication by a constant) see \cite[Th. 2.10 and Th. 2.20]{Folland}.
\end{proof}

To understand the group $A(\gamma)$ we will use its annihilator. 
We define the dual of a locally compact Abelian group $G$ as the topological group $\widehat{G}$ of continuous group homomorphisms (characters) from $G$ to $\mathbf{T}$. In particular one has $\widehat{\mathbf{T}^{N}} \simeq \mathbf{Z}^{N}$ given by the pairing $\langle x, k \rangle = x_1k_1 + \ldots + x_Nk_N$ (see \cite[Cor. 4.7]{Folland}).  
Then the annihilator of a subgroup $H <G$ is the closed subgroup of $\widehat{G}$ (\cite[2.1.1]{Rudin_Fourier}) defined by
	$$H^{\bot} = \lbrace k \in \widehat{G} : \forall x \in H, \langle x, k\rangle = 0 \rbrace.$$

The following result gives a precise description of $A(\gamma)$ using its annihilator.
\begin{lem}\label{Lem_AA=Closure}
	Let $H$ be a subgroup of a locally compact Abelian group $G$, one has $$(H^{\bot})^{\bot} = \overline{H}.$$
	
	In particular for $\gamma = (\gamma_{1}, \ldots, \gamma_{N}) \in \mathbf{R}^{N}$, 
	the annihilator of the group $A(\gamma)$ is
	$$A(\gamma)^{\bot} = \lbrace (k_1,\ldots,k_N) \in \mathbf{Z}^N : k_1\gamma_1 + \ldots k_N\gamma_{N} = 0 \rbrace.$$ 
	
	We note that if $\gamma_{1}, \ldots, \gamma_{N}$ are linearly independent over $\mathbf{Q}$ then $A(\gamma)^{\bot} = \lbrace 0 \rbrace$, hence $A(\gamma) = \mathbf{T}^{N}$. 
\end{lem}

\begin{proof}
	The first point is that there is a natural map: $\Phi : G \rightarrow \widehat{\widehat{G}}$ given by the evaluation:
	for $x \in G$, $k \in \widehat{G}$,
	$$\langle k, \Phi(x) \rangle = \langle x, k \rangle.$$
	By Pontryagin duality theorem, this map is an isomorphism (\cite[Th. 4.31]{Folland}).
	
	In the case $H$ is a closed subgroup of $G$, our first claim is \cite[Prop. 4.38]{Folland}, and it follows from the fact that a non trivial character admits a non trivial value.
	For a general subgroup, it is enough to see that $H^{\bot} = \overline{H}^{\bot}$. This follows from the fact that characters are continuous.

For our particular case,
	the annihilator of the $1$-parameter group $\lbrace y(\gamma_{1},\ldots,\gamma_{N}) : y\in\mathbf{R}\rbrace/\mathbf{Z}^{N}$
	is
	\begin{align*}
	\lbrace (k_1,\ldots,k_N) \in \mathbf{Z}^N &: \forall y \in \mathbf{R},  \langle y(\gamma_{1},\ldots,\gamma_{N}) , (k_1,\ldots,k_N) \rangle = 0_{\mathbf{T}}\rbrace \\
	&= 	\lbrace (k_1,\ldots,k_N) \in \mathbf{Z}^N : \forall y \in \mathbf{R},  y(\gamma_{1}k_1 + \ldots + \gamma_{N}k_N) \in \mathbf{Z} \rbrace \\
	&= 	\lbrace (k_1,\ldots,k_N) \in \mathbf{Z}^N : \gamma_{1}k_1 + \ldots + \gamma_{N}k_N = 0_{\mathbf{R}} \rbrace
	\end{align*}
which gives the conclusion.	
\end{proof}

Finally the proof of Kronecker--Weyl Equidistribution Theorem is a consequence of Poisson's Formula on $\mathbf{T}^{N}$.
	For $G$ a locally compact Abelian group, and $f\in L^{1}(G)$, the Fourier transform of $f$ is a function on $\widehat{G}$ defined by
	$$\widehat{f}(\chi) = \int \langle x,\chi \rangle f(x) \diff\omega_{G}(x).$$	
Then the Poisson formula (\cite[Th. 4.42]{Folland}) for $H$ a closed subgroup of $G$ and $f\in C_{c}(G)$ (continuous with compact support) is the following:
\begin{align*}
\int_H f(x) \diff\omega_{H}(x) 
=  \int_{H^{\bot}} \widehat{f}(\chi) \diff\omega_{H^{\bot}}(\chi)
\end{align*}
	with the suitably normalized Haar measures on $H$ and $H^{\bot}$ (see \cite[Prop. 4.4]{Folland} and \cite[Prop. 4.24]{Folland}). 
In our particular case, the suitably normalized Haar measure on a compact group (such as a sub-torus of $\mathbf{T}^{N}$) is the one whose total mass is $1$, and the normalized Haar measure on a discrete group (such as a sub-lattice of $\mathbf{Z}^N$) is the counting measure. 

We now have all the tools to prove our version of the Kronecker--Weyl equidistribution Theorem.

\begin{proof}[Proof of Theorem~\ref{Th_KW}]
	
	This proof follows \cite{MOF_KW}.
Since $\mathbf{T}^{N}$ is compact, continuous functions on $\mathbf{T}^{N}$ are uniformly continuous, in particular they are limits of polynomials in the uniform convergence topology. Thus it is enough to show that the two terms in \eqref{Form_KW} are equal when $h$ is a trigonometric polynomial.
Then by linearity, it is enough to show it is true for monomials (i.e. characters of $\mathbf{T}^{N}$).
Let $k = (k_1, \ldots, k_N) \in \mathbf{Z}^{N}$,
and $\chi_{k}$ be the associated character (with image in $\mathbf{C}$) of $\widehat{\mathbf{T}^N}$ : $\chi_{k}(x) = e^{2\pi i(k_1x_1 + \ldots + k_Nx_N)}$,

One has
\begin{align*}
\int_{0}^{Y}\chi_{k}(y\gamma_{1},\ldots,y\gamma_{N})\diff y
= \int_{0}^{Y}e^{2\pi i y (k_1\gamma_{1} + \ldots +k_N\gamma_{N})}\diff y 
=\begin{cases}
 Y \text{ if } k_1\gamma_{1} + \ldots +k_N\gamma_{N} = 0 \\
 O(1) \text{ if } k_1\gamma_{1} + \ldots +k_N\gamma_{N} \neq 0.
\end{cases}
\end{align*}

For the left hand side of the equality, we use the Fourier transform and Poisson summation formula.
By orthogonality relations
\begin{align*}
\widehat{\chi_{k}}(\chi) =
\begin{cases}  1 \text{ if  } \chi=\chi_{k}\\
 0 \text{ otherwise.}
 \end{cases}
\end{align*}

By Poisson formula and Lemma~\ref{Lem_AA=Closure}, $A(\gamma)^{\bot}$ is a closed subgroup of $\mathbf{Z}^{N}$ whose normalized Haar measure is the counting measure, we have
\begin{align*}
\int_{A(\gamma)}\chi_{k}(a)\diff\omega_{A(\gamma)} 
=  \sum_{\ell \in A(\gamma)^{\bot}} \widehat{\chi_k}(\chi_\ell)  =
\begin{cases}
 1 \text{ if } k \in A(\gamma)^{\bot}\\
 0 \text{ otherwise.}
\end{cases}
\end{align*}

This concludes the proof.
\end{proof}

\subsection{Approximation of $\psi(f,x)$}\label{Subsec_Approx1}

It is a standard step in proofs of theorems reminiscent of the Prime Number Theorem 
to begin with the study of the associated $\psi$-function~:
$$\psi(f,x) = \sum_{k=1}^{\infty}\sum_{p^{k}\leq x}\left(\sum_{j=1}^{d} \alpha_{j}(p)^k\right) \log p.$$
Note that for $\re(s)>1$, one has
$$-\frac{L'(f,s)}{L(f,s)} = \sum_{k=1}^{\infty}\sum_{p}\left(\sum_{j=1}^{d} \alpha_{j}(p)^k\right)p^{-ks} \log p=: \sum_{n=1}^{\infty}\Lambda_{f}(n)n^{-s}.$$
Then Perron's Formula and integration around the zeros yields an explicit formula for $\psi(f,x)$.

\begin{prop}\label{Prop_decomp_psi}
	Let $L(f,\cdot)$ be an analytic $L$-function.
	One has 
	\begin{align}\label{Eq_splitSum}
	\psi(f,x) + m(L(f,\cdot),1)x  = 
	- \sum_{\substack {L(f,\rho)=0 \\ \lvert\im(\rho)\rvert\leq T}}\frac{x^{\rho}}{\rho} - x^{\beta_{f,0}}\epsilon_{f}(x,T) + 
	O\left(d\left(\log(\mathfrak{q}(f)x^{d})\right)^2 \right)
	\end{align}
	where the function $\epsilon_{f}(x,T)$ satisfies
	\begin{equation}\label{Bound_epsilon}
	\int_{2}^{Y}\lvert \epsilon_{f}(e^{y},T)\rvert^{2} \diff y \ll Y\frac{d^{2}\left(\log(\mathfrak{q}(f)T)\right)^2}{T} + \frac{d^{2}\log(\mathfrak{q}(f)T)^3}{T}
	\end{equation}
	with an absolute implicit constant.
\end{prop}

\begin{proof}
	
	Using Perron's Formula as in \cite[Cor. 5.3]{MV}.
	we obtain a main term 
	$$\frac{1}{2i\pi}\int_{c-iX}^{c+iX} -\frac{L'(f,s)}{L(f,s)} x^{s}\frac{\diff s}{s}$$
	where we choose $c=1+\frac{1}{\log x}$.	
	Using Cauchy's residue Theorem, 
	we write this integral as a sum over the zeros and poles of $L(f,s)$ and an integral that goes on the left ot the critical strip (that can be bounded using bounds on the logarithmic derivative of the $L$-function close to the critical strip, see \cite[Prop. 5.27(2)]{IK}) 
	we obtain 
	\begin{equation}\label{Eq_Perron}
	\psi(f,x) + m(L(f,\cdot),1)x  = -\sum_{\substack {L(f,\rho)=0 \\ \lvert\im(\rho)\rvert\leq X}}\frac{x^{\rho}}{\rho}   + 
	O\left(d\log x + d\frac{x}{X}\left((\log x)^2 + \log(\mathfrak{q}(f)X^{d})\right) + \left(\log(\mathfrak{q}(f)X^{d})\right)^2\right).
	\end{equation}
	with an absolute implicit constant, see also \cite[Chap. 5, Ex. 7]{IK}.
	
	Taking $X=x$ and cutting the sum at $T\leq x$, we obtain
	\begin{equation}\label{Eq_ExplicitFormula_x}
		\psi(f,x) + m(L(f,\cdot),1)x  = -\left\lbrace \sum_{\substack {L(f,\rho)=0 \\ \lvert\im(\rho)\rvert\leq T}} + \sum_{\substack {L(f,\rho)=0 \\ T<\lvert\im(\rho)\rvert \leq x}} \right\rbrace\frac{x^{\rho}}{\rho}  + 
		O\left(d\left(\log(\mathfrak{q}(f)x^{d})\right)^2 \right).
		\end{equation}
		with an absolute implicit constant.

The first sum is our main term, we bound the second moment of the second sum.
	Define
	\begin{equation}\label{Def_epsilon}
	\epsilon_{f}(x,T):= x^{-\beta_{f,0}}\sum_{\substack{\rho=\beta + i\gamma \\ L(f,\rho)=0 \\ T<\lvert\gamma\rvert \leq x}}\frac{x^{\beta +i\gamma}}{\beta+ i\gamma}.
	\end{equation}
The bound given in Proposition~\ref{Prop_decomp_psi} follows from a generalization of \cite[Lem. 3.3]{FioEC} and \cite[Lem. 2.2]{RS}.

\end{proof}

As we do not assume that the Riemann Hypothesis holds, the sum in (\ref{Eq_splitSum}) is not obviously an almost periodic function in the sense of \cite{ANS}. 
We now decompose this sum to highlight the main term and bound the error.

\begin{lem}\label{Lm_beta0}
Let $L(f,s)$ be an analytic $L$-function, and let $T>2$ be fixed.
Define 
$$\beta_{f,T}= \sup\lbrace \re(\rho) : L(f,\rho)=0, \lvert\im(\rho)\rvert\leq T, \re(\rho)<\beta_{f,0} \rbrace.$$
One has
	$$x^{-\beta_{f,0}}\sum_{\substack{ \rho \\ L(f,\rho)=0 \\ \im(\rho)\leq T}}\frac{x^{\rho}}{\rho} = 
	\sum_{\substack{ \gamma\leq T \\ L(f,\beta_{f,0} + i\gamma)=0}}\frac{x^{i\gamma}}{\beta_{f,0} + i\gamma} + O\left(x^{\beta_{f,T}-\beta_{f,0}}\left(\log(\mathfrak{q}(f)T)\right)^2\right).$$
\end{lem}

\begin{Rk}
We use the conventions: $\sup \emptyset = - \infty$ and for $x>0$ one has $x^{-\infty} = 0$.
\end{Rk}

\begin{proof}
Using \cite[Prop. 5.7.(1)]{IK}, we write
\begin{align*}
        x^{-\beta_{f,0}}\sum_{\substack { \re(\rho) < \beta_{f,0} \\ \lvert\im(\rho)\rvert\leq T}}\frac{x^{\rho}}{\rho}
        \ll  x^{\beta_{f,T}-\beta_{f,0}} \sum_{\substack{ \re(\rho) < \beta_{f,0} \\ \lvert\im(\rho)\rvert\leq T}}\frac{1}{\lvert \rho \rvert} 
        \ll  x^{\beta_{f,T}-\beta_{f,0}}\left(\log(\mathfrak{q}(f)T)\right)^{2}.
	\end{align*}
	 The implicit constant is absolute.
\end{proof}

\subsection{Back to $E_{f}(x)$}\label{Subsec_ApproxFin}

The study for $\psi(f,x)$ is now almost settled.
However $E_{f}(x)$ contains another term of potential equal interest.

\begin{lem}\label{Lm_f2}
Let $L(f,s)$ be an analytic $L$-function, 
one has
\begin{align}
\theta(f,x) := \sum_{p\leq x}\lambda_{f}(p)\log p  = \psi(f,x) +m(L(f^{(2)},\cdot),1)x^{\frac{1}{2}} + o_{f}(x^{\frac{1}{2}}).
\end{align}
\end{lem}

\begin{proof}
The Ramanujan--Petersson Conjecture and the Prime Number Theorem yield
$$\sum_{p\leq x}\lambda_{f}(p)\log p = \psi(f,x) 
- \sum_{p^{2}\leq x}\left(\sum_{j=1}^{d} \alpha_{j}(p)^2\right) \log p
+ O(dx^{1/3}).$$
To evaluate the second term, we use Wiener--Ikehara's Tauberian Theorem for the function
$\frac{L'(f^{(2)},s)}{L(f^{(2)},s)}$ (see e.g. \cite[II.7.5]{Tenenbaum}). 
According to Definition \ref{Def_Lfunc}(\ref{Hyp_L2isLfunct}), this function extends meromorphically to the region $\re(s) \geq 1$, with no poles except a simple pole at $s = 1$ with residue $-m(L(f^{(2)},\cdot),1)$.
We obtain
$$\sum_{p^{2}\leq x}\left(\sum_{j=1}^{d} \alpha_{j}(p)^2\right) \log p 
= -m(L(f^{(2)},\cdot),1)\sqrt{x} + o_{f}(\sqrt{x}).$$
\end{proof}

Finally, using Stieltjes integral, we write $E_{f}(x) = \frac{\log x}{x^{\beta_{f,0}}} \int_{2}^{x} \frac{\diff (\theta(f,t) + m(L(f,\cdot),1)t)}{\log t}$.
Using integration by parts we obtain
\begin{multline*}
 x^{\beta_{f,0}} E_{f}(x) =  \psi(f,x) + m(L(f,\cdot),1)x + m(L(f^{(2)},\cdot),1)x^{\frac{1}{2}} \\ +
 O\left(\log x  \int_{2}^{x} \frac{\psi(f,t) + m(L(f,\cdot),1)t}{t(\log t)^{2}}\diff t \right)
 + o_{f}(x^{\frac{1}{2}}). 
\end{multline*}
We use again an integration by parts to evaluate the $O$ term.
From \eqref{Eq_Perron}, after integrating and taking $X=x$, we have
$$ \int_{2}^{x} \left(\psi(f,t) + m(L(f,\cdot),1)t\right)\diff t =
 - \sum_{\substack {L(f,\rho)=0 \\ \lvert\im(\rho)\rvert\leq x}}\frac{x^{\rho+1}}{\rho(\rho +1)}  + O_{f}(x(\log x)^{2}).$$
This series converges absolutely, so we can permute the limits.
We deduce that the $O$ term is $O_{f}(x^{\beta_{f,0}}/\log x)$.
Hence we have 
\begin{align}\label{Form_E_psi}
E_{f}(x) =  \frac{1}{ x^{\beta_{f,0}} }\left(\psi(f,x) + m(L(f,\cdot),1)x\right) + m(L(f^{(2)},\cdot),1)x^{\frac{1}{2}-\beta_{f,0}} 
 + o_{f}(1).
\end{align}

\subsection{Existence of the limiting distribution}\label{sub_ExistenceLimDist}

We can now prove Proposition~\ref{Prop_LimOfDist}. 
In particular we prove the first point of Theorem~\ref{Th_DistLim}: the existence of the limiting distribution for the function $E_{f}$.

Define (see Proposition \ref{Prop_LimOfDist}),
\begin{equation*}
G_{f,T}(x) = \frac{m(L(f,\cdot),\beta_{f,0})}{\beta_{f,0}} + m(L(f^{(2)},\cdot),1)\delta_{\beta_{f,0}=\frac{1}{2}}
 -\sum_{ \gamma\in\mathcal{Z}_{f}^{*}(T)}2\re\left(m(L(f,\cdot),\beta_{f,0} + i\gamma)\frac{x^{i\gamma}}{\beta_{f,0} + i\gamma}\right).
\end{equation*}
We use (\ref{Form_E_psi}) where we evaluate $\psi(f,x)x^{-\beta_{f,0}}$ using Proposition~\ref{Prop_decomp_psi}
and Lemma~\ref{Lm_beta0}.
We can now write
\begin{align}\label{Form_finale}
E_{f}(x)= G_{f,T}(x)
+ O_{f}\left(x^{\beta_{f,T}-\beta_{f,0}}(\log T)^2\right) 
- \epsilon_{f}(x,T)
 + o(1)
\end{align}
where the second term vanishes if the Riemann Hypothesis is satisfied.
We first prove that the real function $G_{f,T}$ admits a logarithmic distribution.

\begin{lem}\label{Lm_LimDistG}
Let $T>2$ fixed.
Then $G_{f,T}$ admits a limiting logarithmic distribution $\mu_{f,T}$.
\end{lem}

\begin{proof} 
This follows from the generalized Kronecker--Weyl Theorem (see \cite[Lem. 2.3]{RS} or \cite[Prop. 2.4]{ANS}, and Theorem~\ref{Th_KW} for a more detailed proof).
We write $\mathcal{Z}_{f}^{*}(T) =\lbrace \gamma_{1},\ldots,\gamma_{N(T)}\rbrace$.
Let $g:\mathbf{R}\rightarrow\mathbf{R}$ be a bounded Lipschitz continuous function,
one can associate to $g$ the continuous function on $\mathbf{T}^{N(T)}$ defined by
\begin{align}\label{form_tilde}
\tilde{g}(t)=  g\left(m_{f} - 2\re\left(\sum_{k=1}^{N(T)}\frac{e^{2i\pi t_{k}}}{\beta_{f,0}+i\gamma_{k}}\right)\right).
\end{align}
One has 
$$\int_{2}^{Y} g(G_{f,T}(e^{y}))\diff y = \int_{2}^{Y} \tilde{g}\left(\frac{\gamma_{1}}{2\pi}y,\ldots,\frac{\gamma_{N(T)}}{2\pi}y\right)\diff y.$$
Then we see that the measure $\mu_{f,T}$ is the push-forward measure of the normalized Haar measure on the closure of 
$\lbrace (\frac{\gamma_{1}}{2\pi}y,\ldots,\frac{\gamma_{N(T)}}{2\pi}y) : y\in \mathbf{R}\rbrace/\mathbf{Z}^{N(T)}$ in $\mathbf{T}^{N(T)}$. 
\end{proof}

Next using (\ref{Form_finale}) we see that $E_{f}$ is a $B^{2}$-almost periodic function, hence by \cite[Th. 2.9]{ANS} it admits a limiting logarithmic distribution (see also \cite[Th. 1(e)]{KaczoRama}).
In particular the proof uses the following inequality:
for $g$ a continuous $C_{g}$-Lipschitz bounded function, and $T>0$ fixed,
\begin{multline}\label{Formule_BoundsLimDist}
\int_{\mathbf{R}}g \diff\mu_{f,T} + O_{f}\left(C_{g} \frac{\log T }{\sqrt{T}} \right) \leq
\liminf_{Y\rightarrow\infty} \frac{1}{Y}  \int_{2}^{Y}g(E_{f}(e^{y})) \diff y \\
\leq \limsup_{Y\rightarrow\infty} \frac{1}{Y}  \int_{2}^{Y}g(E_{f}(e^{y})) \diff y  \leq
\int_{\mathbf{R}}g \diff \mu_{f,T} + O_{f}\left(C_{g} \frac{\log T}{\sqrt{T}} \right).
\end{multline}
The proof of Proposition~\ref{Prop_LimOfDist} follows from Helly's Selection Theorem.

\subsection{Mean and Variance}\label{Subsec_MeanVar}
We complete the proof of Theorem~\ref{Th_DistLim} by studying the decay of $\mu_{f}$ at infinity and computing its mean and variance.

Using (\ref{Formule_BoundsLimDist}) and information on the support of $\mu_{f,T}$ we can show that $\mu_{f}$ has exponential decay.
\begin{lem}\label{Lm_expDecay}
There exists a positive constant $c(f)$ depending only on $f$ such that
$$\mu_{f}(\mathbf{R} - [m_{f} - R, m_{f} + R])\ll_{f} e^{-c(f)\sqrt{R}}.$$
\end{lem}
\begin{proof}
One has 
\begin{align*}
\left\lvert\sum_{ \gamma\in\mathcal{Z}_{f}^{*}(T)}2\re\left(m(L(f,\cdot),\beta_{f,0} + i\gamma)\frac{x^{i\gamma}}{\beta_{f,0} + i\gamma}\right)\right\rvert 
\leq  \sum_{ \gamma\in\mathcal{Z}_{f}(T)}2\frac{1}{\lvert \beta_{f,0} + i\gamma\rvert} 
\ll_{f} (\log T)^{2}
\end{align*}
thanks to the fact that the completed $L$-function is of order $1$ (see Definition~\ref{Def_Lfunc}(\ref{Hyp_FunctEquation})).
Therefore the function $G_{f,T}(e^{y})$ is bounded.
We deduce that the measure $\mu_{f,T}$ has compact support included in 
$[m_{f} - c_{1}(\log T)^2 , m_{f} + c_{1}(\log T)^2]$ for some constant $c_{1}$ depending on $f$.
Using (\ref{Formule_BoundsLimDist}), we have
$$\mu_{f}(\mathbf{R} - [m_{f} - c_{1}(\log T)^2 , m_{f} + c_{1}(\log T)^2]) = O_{f}\left( \frac{\log T}{\sqrt{T}} \right).$$
For $R=c_{1}(\log T)^2$ the result follows.
\end{proof}

The measure $\mu_{f}$ has exponential decay at infinity, hence it has finite moments.
The values for the mean and variance given in Theorem \ref{Th_DistLim} follow from computations for $\mu_{f,T}$, letting $T$ be arbitrarily large as we now explain, the proofs follow the ideas of \cite{FioEC}.
Let $T\geq2$ be fixed, one has
\begin{align*}
\int_{\mathbf{R}}t\diff\mu_{f,T}
&= \lim_{Y\rightarrow\infty} \frac{1}{Y} \int_{2}^{Y} \left(m_{f}  -\sum_{ \gamma\in\mathcal{Z}_{f}(T)}2\re\left(\frac{e^{iy\gamma}}{\beta_{f,0} + i\gamma}\right)\right) \diff y \\
&= m_{f} - \lim_{Y\rightarrow\infty} \frac{1}{Y}  O\left(\sum_{ \gamma\in\mathcal{Z}_{f}(T)}\frac{1}{\lvert\beta_{f,0} + i\gamma\rvert \lvert\gamma\rvert} \right)
= m_{f}
\end{align*}
because the sum over $\mathcal{Z}_{f}^{*}(T)$ is finite.
Taking the limit as $T \rightarrow\infty$ the assertion on the average value of $\mu_{f}$ is proved.

For the computation of the variance, we cannot use the linearity anymore, we go back to the general case.
Set 
$$G_{\mathcal{S},T}(x) = m_{\mathcal{S}} - \sum_{\gamma\in \mathcal{Z}^{*}_{\mathcal{S}}(T)} 2\re\left(M(\gamma)\frac{x^{i\gamma}}{\beta_{\mathcal{S},0} + i\gamma}\right)$$
where for $\gamma$ in $\mathcal{Z}_{\mathcal{S}}^{*}$, we denote $M(\gamma) = \sum_{f\in\mathcal{S}}a_{f} m(L(f,\cdot),\beta_{\mathcal{S},0} + i\gamma)$.
Then
\begin{align*}
\int_{\mathbf{R}}\lvert t - m_{\mathcal{S}}\rvert^{2}\diff\mu_{\mathcal{S},T}
&= \lim_{Y\rightarrow\infty} \frac{1}{Y}\int_{2}^{Y} \left\lvert\sum_{\gamma\in \mathcal{Z}^{*}_{\mathcal{S}}(T)} \left(\frac{M(\gamma)e^{iy\gamma}}{\beta_{0} + i\gamma} + \frac{M(-\gamma)e^{-iy\gamma}}{\beta_{0} - i\gamma}\right)\right\rvert^{2} \diff y \\
&= \lim_{Y\rightarrow\infty} \frac{1}{Y}\sums_{\gamma,\lambda}\frac{M(\gamma)\overline{M(\lambda)}}{(\beta_{0} + i\gamma)(\beta_{0} - i\lambda)}\int_{2}^{Y}e^{i(\gamma-\lambda)y}\diff y
\end{align*}
where the $\gamma$, $\lambda$ are in the index set $\mathcal{Z}_{\mathcal{S}}^{*}(T)\cup (- \mathcal{Z}_{\mathcal{S}}^{*}(T))$ (counted without multiplicities).
The diagonal term $\lambda=\gamma$ is the main term. 
The off-diagonal term vanishes as $Y\rightarrow\infty$ when $T$ is fixed.
One has 
\begin{align*}
\sums_{\gamma\neq\lambda}\frac{M(\gamma)\overline{M(\lambda)}}{(\beta_{\mathcal{S},0} + i\gamma)(\beta_{\mathcal{S},0} - i\lambda)}\int_{2}^{Y}e^{i(\gamma-\lambda)y}\diff y
&= O\left(  \sums_{\gamma\neq\lambda}\frac{\lvert M(\gamma)\rvert\lvert M(\lambda)\rvert}{\lvert\gamma\rvert\lvert\lambda\rvert}\min(Y,\lvert \gamma-\lambda\rvert^{-1}) \right).
\end{align*}
Using \cite[Lem. 2.6]{FioEC},
we deduce that
\begin{equation*}
\int_{\mathbf{R}}\lvert t - m_{\mathcal{S}}\rvert^{2}\diff\mu_{\mathcal{S}} = 
\sum_{\gamma\in\mathcal{Z}^{*}_{\mathcal{S}}}
\frac{2\lvert M(\gamma)\rvert^{2}}{\lvert \beta_{\mathcal{S},0} + i\gamma \rvert^{2}}.
\end{equation*}
This concludes the proof of Theorem~\ref{Th_DistLim}.

%% file: AddHyp.tex
It is clear from the proof of Theorem \ref{Th_DistLim} 
that the properties of the set of non trivial $L$-function zeros of largest real part are related to the properties of $\mu_{\mathcal{S}}$.
In this section we investigate in more details what can be inferred from additional hypotheses on the zeros.

\subsection{Existence of the bias and regularity of the distribution}\label{sub_Indep}

We show that the existence of self-sufficient zeros in $\mathcal{Z}_{\mathcal{S}}^{*}$ 
gives properties of smoothness for $\mu_{\mathcal{S}}$. 
Such results were previously obtained (e.g. in \cite{RS}) conditionally on LI.

The first point we need to address is to understand more precisely the measures $\mu_{\mathcal{S},T}$ and for this we need to give more precisions on the sub-tori $A(\gamma_{1}, \ldots, \gamma_{N(T)})$.
We prove that such a sub-torus can be decomposed into products of sub-tori if there is some linear independence between the zeros.
\begin{prop}\label{Prop_Indep}
	Let $\gamma_1, \ldots, \gamma_M, \lambda_{1}, \ldots, \lambda_N \in \mathbf{R}$ be arbitrary real numbers satisfying
	$$\langle \gamma_1, \ldots, \gamma_M \rangle_{\mathbf{Q}} \cap \langle \lambda_{1}, \ldots, \lambda_N \rangle_{\mathbf{Q}} = \lbrace 0 \rbrace.$$
	Then the topological closure of the $1$-parameter group $\lbrace y(\gamma_{1},\ldots,\gamma_{M}, \lambda_{1}, \ldots, \lambda_N) : y\in\mathbf{R}\rbrace/\mathbf{Z}^{M+N}$ in the $M+N$-dimensional torus $\mathbf{T}^{M+N}$ 
	is the Cartesian product of the topological closure of the  $1$-parameter group $\lbrace y(\gamma_{1},\ldots,\gamma_{M}) : y\in\mathbf{R}\rbrace/\mathbf{Z}^{M}$ in the $M$-dimensional torus $\mathbf{T}^{M}$ 
	and the  topological closure of the $1$-parameter group $\lbrace y(\lambda_{1}, \ldots, \lambda_N) : y\in\mathbf{R}\rbrace/\mathbf{Z}^{N}$ in the $N$-dimensional torus $\mathbf{T}^{N}$: 
	$$A(\gamma,\lambda) = A(\gamma) \times A(\lambda).$$
	
	In particular, the Haar measure over $A(\gamma,\lambda)$ is the product of the Haar measures over the sub-tori: 
	$$\diff\omega_{A(\gamma,\lambda)} = \diff\omega_{A(\gamma)}\diff\omega_{A(\lambda)}.$$
\end{prop} 

A corollary of this result is that we can deal with independent sets independently when taking the limiting distribution (for example we can use Fubini Theorem). 

\begin{proof}
		Using Lemma~\ref{Lem_AA=Closure}, we write that 
	$A(\gamma,\lambda) = A(\gamma_{1},\ldots,\gamma_{M},\lambda_{1},\ldots,\lambda_N)$ is the annihilator of the annihilator of the $1$-parameter group
	$\lbrace y(\gamma_{1},\ldots,\gamma_{M},\lambda_{1},\ldots,\lambda_N) : y \in \mathbf{R}\rbrace / \mathbf{Z}^{M+N}$ in the group $\mathbf{T}^{M+N}$. 
	
	So we first need to determinate the annihilator:
	\begin{multline*}
	\left(\lbrace y(\gamma_{1},\ldots,\gamma_{M},\lambda_{1},\ldots,\lambda_N) : y \in \mathbf{R}\rbrace / \mathbf{Z}^{M+N}\right)^{\bot} \\
	= 	\lbrace (k_1,\ldots,k_M,\ell_1,\ldots, \ell_N) \in \mathbf{Z}^{M+N} : \gamma_{1}k_1 + \ldots + \gamma_{N}k_M + \lambda_{1}\ell_1 + \ldots \lambda_N\ell_N = 0 \rbrace \\
	= \lbrace (k_1,\ldots,k_M,\ell_1,\ldots, \ell_N) \in \mathbf{Z}^{M+N} : \gamma_{1}k_1 + \ldots + \gamma_{N}k_M =0 \text{ and } \lambda_{1}\ell_1 + \ldots \lambda_N\ell_N = 0 \rbrace \\
	= A(\gamma_1,\ldots,\gamma_M)^{\bot} \times A(\lambda_{1}, \ldots, \lambda_N)^{\bot},
	\end{multline*}
	The second equality follows from the linear independence. 
	Indeed if for some $(k_1,\ldots,k_M,\ell_1,\ldots, \ell_N) \in \mathbf{Z}^{M+N}$
	one has $ \gamma_{1}k_1 + \ldots + \gamma_{N}k_M + \lambda_{1}\ell_1 + \ldots \lambda_N\ell_N = 0$
	then $ \gamma_{1}k_1 + \ldots + \gamma_{N}k_M = \lambda_{1}\ell_1 + \ldots \lambda_N\ell_N = 0$.
	The third equality is Lemma~\ref{Lem_AA=Closure}.
	
	We conclude by taking the annihilator (thanks to Lemma~\ref{Lem_AA=Closure} once again) and the fact that the annihilator of a product is the product of annihilators (see Lemma~\ref{Lem_AnProduct}). \end{proof}
	
	\begin{lem}\label{Lem_AnProduct}
		Let $G_1$ and $G_2$ be two locally compact Abelian groups,
		let $H_1$ be a subgroup of $G_1$ and $H_2$ be a subgroup of $G_2$.
		Then $H_1\times H_2$ is a subgroup of $G_1\times G_2$ and its annihilator is a subgroup of $\widehat{G_1}\times \widehat{G_2}$ given by the product
		$$\left(H_1\times H_2\right)^{\bot} = H_1^{\bot}\times H_2^{\bot}.$$
	\end{lem}
	
	\begin{proof}
		For the fact that $\widehat{G_1\times G_2} = \widehat{G_1}\times\widehat{G_2}$,
		see \cite[Prop. 4.6]{Folland}.
		For $(k_1,k_2)\in  H_1^{\bot}\times H_2^{\bot}$ and $(x_1,x_2) \in H_1\times H_2$, one has
		$$\langle (x_1,x_2) , (k_1,k_2) \rangle = \langle x_1 , k_1 \rangle +  \langle x_2 , k_2 \rangle =0$$
		hence $H_1^{\bot}\times H_2^{\bot} \subset \left(H_1\times H_2\right)^{\bot}$.
		Now take $(k_1,k_2)\in  \left(H_1\times H_2\right)^{\bot}$,
		for all $(x_1,x_2) \in H_1\times H_2$ one has
		$$ \langle x_1 , k_1 \rangle +  \langle x_2 , k_2\rangle =\langle (x_1,x_2) , (k_1,k_2) \rangle  =0.$$
		In particular ($0 \in H_1$), for all $x_2 \in H_2$,
		$$\langle x_2 , k_2\rangle =  \langle 0 , k_1 \rangle +  \langle x_2 , k_2\rangle =\langle (0,x_2) , (k_1,k_2) \rangle  =0,$$
		hence $k_{2} \in H_2^{\bot}$.
		Similarly, $k_1\in H_1^{\bot}$ and the proof is complete.
	\end{proof}

Our main contribution in the following result is that we get the existence of the logarithmic density $\delta(\mathcal{S})$ (as defined in Definition~\ref{Def_logdens}) under weaker hypotheses than LI.
We use the concept of self-sufficient zeros introduced by Martin and Ng in \cite{MartinNg}.

\begin{defi}\label{Def_selfsuff}\begin{enumerate}
		\item An ordinate $\gamma\in\mathcal{Z}_{\mathcal{S}}^{*}$ is self-sufficient if it is not in the $\mathbf{Q}$-span of $\mathcal{Z}_{\mathcal{S}}^{*}-\lbrace \gamma\rbrace$.
		\item For $U,V>0$, we say that an ordinate $\gamma\in\mathcal{Z}_{\mathcal{S}}^{*}(U)$ is $(U,V)$-self-sufficient if it is not in the $\mathbf{Q}$-span of $\mathcal{Z}_{\mathcal{S}}^{*}(V)-\lbrace \gamma\rbrace$.
	\end{enumerate}
\end{defi}

Using this concept we prove conditional results for the regularity of the distribution $\mu_{\mathcal{S}}$.

\begin{theo}\label{Th_withLI}
	Suppose that the set $\mathcal{S}$ satisfies the conditions of Theorem~\ref{Th_DistLim}.
	\begin{enumerate}
		\item\label{Th1_Tselfsuff} 	Suppose that there exists $\epsilon >0$ and $T_{\epsilon} >0$
		such that for every $T\geq T_{\epsilon}$ the set $\mathcal{Z}_{\mathcal{S}}^{*}(T^{\frac{1}{2} - \epsilon})$ 
		contains a $(T^{\frac{1}{2} - \epsilon},T)$-self-sufficient zero $\gamma_{T}$.
		Then $\delta(\mathcal{S})$ exists.
		\item\label{Th1_LI1} Suppose $\mathcal{Z}^{*}_{\mathcal{S}}$ contains at least one self-sufficient element, 
		then the distribution $\mu_{\mathcal{S}}$ is continuous (\textit{i.e.} $\mu_{\mathcal{S}}$ assigns zero mass to finite sets).
		\item\label{Th2_LI3} Suppose $\mathcal{Z}^{*}_{\mathcal{S}}$ contains three or more self-sufficient elements, 
		then the distribution $\mu_{\mathcal{S}}$ admits a density $\phi\in L^{1}$ (\textit{i.e.} $\diff\mu_{\mathcal{S}}(x) = \phi(x)\diff x$).
		\item\label{Th3_enoughIndep} Suppose that the set 
		$\lbrace\gamma\in\mathcal{Z}_{\mathcal{S}}^{*} : \gamma \text{ self-sufficient} \rbrace$ is infinite, 
		then the distribution $\mu_{\mathcal{S}}$ admits a density $\phi$ which is in the Schwartz space of indefinitely differentiable and rapidly decreasing functions.
	\end{enumerate}
\end{theo}

\begin{Rk}
This improves some results of \cite{RS} which are obtained under the Grand Simplicity Hypothesis (also called LI).
In loc. cit. Rubinstein and Sarnak obtain Theorem \ref{Th_withLI} under LI, 
i.e. assuming that $\mathcal{Z}_{\mathcal{S}}$ is linearly independent over $\mathbf{Q}$.
In Theorem~\ref{Th_withLI} the Riemann Hypothesis is not needed, and hypotheses in (\ref{Th1_Tselfsuff})--(\ref{Th3_enoughIndep}) are ordered by increasing strength but are all weaker than LI.
In particular Theorem~\ref{Th_withLI}(\ref{Th1_Tselfsuff}) gives a condition for the logarithmic density of the set 
$\lbrace x\geq 2: E_{\mathcal{S}}(x)\geq 0\rbrace$ to exist, 
where the function $E_{\mathcal{S}}$ is as defined in Theorem~\ref{Th_DistLim}.
\end{Rk}

\begin{proof}[Proof of Theorem~\ref{Th_withLI}(\ref{Th1_Tselfsuff})]
Fix $T\geq T_{\epsilon}$.
	Following \cite[Part 3.1]{RS}, we compute the Fourier transform of $\mu_{\mathcal{S},T}$.
	We obtain
\begin{align*}
\hat{\mu}_{\mathcal{S},T}(\xi) &= \int_{A_{T}}\exp\left( -i\xi \left( m_{\mathcal{S}} - 2\re\left(\sum_{\gamma\in\mathcal{Z}_{\mathcal{S}}^{*}(T)}\frac{M(\gamma)e^{2i\pi t_{\gamma}}}{\beta_{\mathcal{S},0}+i\gamma}\right) \right) \right)\diff t \\
&= e^{-im_{\mathcal{S}}\xi}\int_{A_{T}}\prod_{\gamma\in\mathcal{Z}_{\mathcal{S}}^{*}(T)}\exp\left(i\xi 2\re\left(M(\gamma)\frac{e^{2i\pi t_{\gamma}}}{\beta_{\mathcal{S},0}+i\gamma}\right) \right)\diff t
\end{align*}
where $A_{T}$ is the closure of $\lbrace (\gamma_{1}y,\ldots,\gamma_{N(T)}y) : y\in \mathbf{R}\rbrace/\mathbf{Z}^{N(T)}$ in $\mathbf{T}^{N(T)}$.
The ordinate $\gamma_{T}$ is self-sufficient in $\mathcal{Z}_{\mathcal{S}}^{*}(T)$, hence by Proposition~\ref{Prop_Indep}, one can write $A_{T}= \mathbf{T}\times A_{T}'$ and separate the integral:
\begin{multline*}
\hat{\mu}_{\mathcal{S},T}(\xi) = e^{-im_{\mathcal{S}}\xi}\int_{A'_{T}}\prod_{\gamma\in\mathcal{Z}_{\mathcal{S}}^{*}(T)-\lbrace\gamma_{0}\rbrace}\exp\left(i\xi 2\re\left(M(\gamma)\frac{e^{2i\pi t_{\gamma}}}{\beta_{\mathcal{S},0}+i\gamma}\right) \right)\diff t \\
\int_{\mathbf{T}}\exp\left(i\xi 2\re\left(M(\gamma_{T})\frac{e^{2i\pi \theta}}{\beta_{\mathcal{S},0}+i\gamma_{T}}\right) \right)\diff\theta.
\end{multline*}
The integral over $\mathbf{T}$ is a $0$-th Bessel function of the first kind equal to
$J_{0}\left(\left\lvert\frac{2\xi M(\gamma_{T})}{\beta_{\mathcal{S},0} + i\gamma_{T}} \right\rvert\right).
$
Using properties of the Bessel function (see e.g. \cite{Watson}) and 
the fact that the first integral on the right hand side is bounded from above by $1$, 
one can bound the Fourier transform of $\mu_{\mathcal{S},T}$:
\begin{align}\label{Bound_1T-Aut}
\lvert \hat{\mu}_{\mathcal{S},T}(\xi)\rvert \leq \left\lvert J_{0}\left(\left\lvert\frac{2\xi M(\gamma_{T})}{\beta_{\mathcal{S},0} + i\gamma_{T}} \right\rvert\right) \right\rvert 
\leq \min\left( 1, \sqrt{\left\lvert\frac{\beta_{\mathcal{S},0} + i\gamma_{T}}{ \pi\xi M(\gamma_{T})}\right\rvert}  \right).
\end{align}

Let us come back to the existence of $\delta(\mathcal{S})$.
We want to prove that the limits $$\limsup\frac{1}{Y}\int_{2}^{Y}\mathbf{1}_{\geq 0}(E_{\mathcal{S}}(e^{y}))\diff y \qquad \text{ and } \qquad \liminf\frac{1}{Y}\int_{2}^{Y}\mathbf{1}_{\geq0}(E_{\mathcal{S}}(e^{y}))\diff y$$ coincide.
We write $\mathbf{1}_{\geq 0} = g_{n} + (\mathbf{1}_{\geq 0} - g_{n})$ where $g_n$ is the $n$-Lipschitz function satisfying
$$g_{n}(x)=
\begin{cases}
0  \mbox{ if $x\leq -1/2n$,}\\
1  \mbox{ if $x\geq 1/2n$,}\\
nx  + 1/2  \mbox{ otherwise.}
\end{cases}
$$
The functions $g_n$ and $\lvert \mathbf{1}_{\geq 0} - g_{n}\rvert$ are bounded, continuous, $n$-Lipschitz.
Hence
\begin{align*}
\lim_{Y\rightarrow\infty}\frac{1}{Y}\int_{2}^{Y} g_{n}(G_{\mathcal{S},T}(e^{y}))\diff y = \int_{\mathbf{R}}g_{n}(t)\diff\mu_{\mathcal{S},T}(t),
\end{align*}
and we have the same result if we replace $g_{n}$ by $\lvert \mathbf{1}_{\geq 0} - g_{n}\rvert$.
Taking $T$ arbitrarily large, we approximate the limiting distribution $\mu_{\mathcal{S}}$.
Precisely one has
\begin{multline}\label{Expr_Indic}
\int_{\mathbf{R}}g_{n}(t)\diff\mu_{\mathcal{S},T} - \int_{\mathbf{R}}\lvert \mathbf{1}_{\geq 0} - g_{n}\rvert(t)\diff\mu_{\mathcal{S},T}  + O_{\mathcal{S}}\left(n \frac{\log T}{\sqrt{T}} \right) \\ 
\leq
\liminf_{Y\rightarrow\infty} \frac{1}{Y}  \int_{2}^{Y}\mathbf{1}_{\geq 0}(E_{\mathcal{S}}(e^{y})) \diff y 
\leq \limsup_{Y\rightarrow\infty} \frac{1}{Y}  \int_{2}^{Y}\mathbf{1}_{\geq 0}(E_{\mathcal{S}}(e^{y})) \diff y \\ 
\leq
\int_{\mathbf{R}}g_{n}(t)\diff\mu_{\mathcal{S},T} + \int_{\mathbf{R}}\lvert \mathbf{1}_{\geq 0} - g_{n}\rvert(t)\diff\mu_{\mathcal{S},T} + O_{\mathcal{S}}\left(n \frac{\log T}{\sqrt{T}} \right).
\end{multline}

Moreover we can bound $\mu_{\mathcal{S},T}(\lvert \mathbf{1}_{\geq 0} - g_{n}\rvert)$.
Using Parseval's formula (see e.g. \cite[Th. VI.2.2]{Katznelson}) we have:
\begin{align*}
\int_{\mathbf{R}}\lvert \mathbf{1}_{\geq 0} - g_{n}\rvert \diff\mu_{\mathcal{S},T} =
\int_{\mathbf{R}} 2n\frac{1-\cos(\xi/2n)}{\xi^{2}}\hat{\mu}_{\mathcal{S},T}(\xi)\diff\xi 
\ll \int_{\lvert\xi\rvert\leq \alpha(n)} \frac{1}{2n}\lvert\hat{\mu}_{\mathcal{S},T}(\xi)\rvert \diff\xi    +     \int_{\lvert\xi\rvert\geq \alpha(n)} \frac{4n}{\xi^{2}}\lvert\hat{\mu}_{\mathcal{S},T}(\xi)\rvert \diff\xi
\end{align*}
for $\alpha(n) <2n$.
Using (\ref{Bound_1T-Aut}), we get:
\begin{align*}
\int_{\mathbf{R}}\lvert \mathbf{1}_{\geq 0} - g_{n}\rvert \diff\mu_{\mathcal{S},T} 
\ll  \frac{2\alpha(n)}{2n}    +     \int_{\lvert\xi\rvert\geq \alpha(n)} \frac{4n\sqrt{\gamma_{T}}}{\lvert\xi\rvert^{5/2}} \diff\xi 
\ll \frac{\alpha(n)}{n} + \frac{n\sqrt{\gamma_{T}}}{\alpha(n)^{3/2}}.
\end{align*}
Choose $n=\sqrt{T^{1-\epsilon}}$ and $\alpha(n) = n^{1-\frac{\epsilon}{3}}$.
Since $\gamma_{T} \leq T^{\frac{1}{2}-\epsilon}$, 
the terms of rest in (\ref{Expr_Indic}) vanish as $T\rightarrow+\infty$.
It ensures that the inferior and superior limits coincide.
\end{proof}

To prove the other points of Theorem~\ref{Th_withLI} we follow the same idea without dependence on $T$.

\begin{proof}[Proof of Theorem~\ref{Th_withLI}(\ref{Th1_LI1})]
The fact that $\mu_{\mathcal{S}}$ is continuous is a consequence of a theorem of Wiener (see e.g. \cite[Th. VI.2.11]{Katznelson}).
A necessary and sufficient condition for $\mu_{\mathcal{S}}$ to be continuous is:
\begin{align}\label{Cond_Wiener}
\lim_{Y\rightarrow\infty}\frac{1}{2Y}\int_{-Y}^{Y}\lvert \hat{\mu}_{\mathcal{S}}(\xi)\rvert^{2}\diff\xi =  0.
\end{align}
In the case $\gamma_{T}= \gamma_{0}$ does not depend on $T$, the bound (\ref{Bound_1T-Aut}) becomes, for all $T>\gamma_{0}$,
\begin{align*}
\lvert \hat{\mu}_{\mathcal{S},T}(\xi)\rvert 
\leq \min\left( 1, \sqrt{\left\lvert\frac{\beta_{\mathcal{S},0} + i\gamma_{0}}{ \pi\xi M(\gamma_{0})}\right\rvert}  \right).
\end{align*}
Letting $T\rightarrow\infty$, the same bound holds for $\hat{\mu}_{\mathcal{S}}$.
In particular Condition~(\ref{Cond_Wiener}) holds. 
\end{proof}

\begin{proof}[Proof of Theorem~\ref{Th_withLI}(\ref{Th2_LI3})]
Let $\gamma_{1} <\gamma_{2} <\gamma_{3}$ be three self-sufficient elements of $\mathcal{Z}^{*}_{\mathcal{S}}$.
Following the lines of the previous proofs, we get that for all $T> \gamma_{3}$ one has 
\begin{align*}
\lvert \hat{\mu}_{\mathcal{S},T}(\xi)\rvert 
\leq \prod_{j=1}^{3}\min\left( 1, \sqrt{\left\lvert\frac{\beta_{\mathcal{S},0} + i\gamma_{j}}{ \pi\xi M(\gamma_{j})}\right\rvert}  \right)
\end{align*}
Letting $T\rightarrow\infty$, the same bound holds for $\hat{\mu}_{\mathcal{S}}$.
In particular one has $\hat{\mu}_{\mathcal{S}}\in L^{1}\cap L^{2}$,
Theorem~\ref{Th_withLI}(\ref{Th2_LI3}) follows by Fourier inversion.
\end{proof}

\begin{proof}[Proof of Theorem~\ref{Th_withLI}(\ref{Th3_enoughIndep})]
As in the previous proofs we can write for all $T>0$, for all $\xi$,
\begin{align*}
\lvert \hat{\mu}_{\mathcal{S},T}(\xi) \rvert \leq \prod_{\substack{\gamma \in \mathcal{Z}^{*}_{\mathcal{S}}(T) \\ \text{self-sufficient}}} \min\left( 1, \sqrt{\left\lvert\frac{\beta_{\mathcal{S},0} + i\gamma}{ \pi\xi M(\gamma)}\right\rvert}  \right)
\end{align*}
We assume that there are infinitely many self-sufficient elements in $\mathcal{Z}^{*}_{\mathcal{S}}$.
For each $n \in \mathbf{N}$ there exists $T_{n} >0$ such that
$$\lvert \lbrace \gamma \in \mathcal{Z}^{*}_{\mathcal{S}}(T_{n}) : \gamma \text{ self-sufficient }\rbrace\rvert \geq 2n +1.$$
Hence there exists a constant $C_{n}$ depending only on $\mathcal{Z}^{*}_{\mathcal{S}}(T_{n})$ such that for every $T\geq T_{n}$, for every $\xi$ large enough (in terms of $T_{n}$) one has
\begin{align*}
\lvert \hat{\mu}_{\mathcal{S},T}(\xi) \rvert \leq \frac{C_{n}}{\lvert\xi\rvert^{n}\sqrt{\lvert \xi\rvert}}.
\end{align*}
Letting $T\rightarrow\infty$, the same bound holds for $\hat{\mu}_{\mathcal{S}}$.
In particular one has $\lvert \xi^{n}\hat{\mu}_{\mathcal{S}}(\xi)\rvert \rightarrow 0$ as $\lvert \xi\rvert \rightarrow \infty$.
By Fourier inversion, we obtain that the density of $\mu_{\mathcal{S}}$ is $n$ times differentiable for all $n\geq0$.

The statement about fast decay is a consequence of the exponential decrease obtained in Theorem~\ref{Th_DistLim}.
\end{proof}

In the previous proofs we have used the decay at infinity of the Bessel $0$-th function $J_{0}$ 
to obtain the bounds for $\hat{\mu}_{\mathcal{S}}$.
Using the theory of oscillatory integrals we can deduce the decay of other functions.
We can in fact have condition~(\ref{Cond_Wiener}) under a weaker hypothesis, that leads us to Theorem~\ref{Th_SomeSelfSufficience}.

\begin{proof}[Proof of Theorem~\ref{Th_SomeSelfSufficience}]
The hypothesis and Proposition~\ref{Prop_Indep} imply that for every $T\geq \max\lbrace\lambda_{j} : 1\leq j\leq N \rbrace$, 
the sub-torus $A_{T}$ given by
the topological closure of the $1$-parameter group 
$\left\lbrace y \left(\gamma_{1},\ldots,\gamma_{N(T)}\right) : y\in\mathbf{R} \right\rbrace/\mathbf{Z}^{N(T)}$
can be written as a Cartesian product $A(\lambda_{1},\ldots,\lambda_{N}) \times A'_{T}$
where the two components are respectively the sub-torus associated with the set $\lbrace \lambda_{1},\ldots,\lambda_{N}\rbrace$ 
and with the set $\mathcal{Z}_{\mathcal{S}}^{*}(T) - \lbrace \lambda_{1},\ldots,\lambda_{N}  \rbrace$ (see Theorem~\ref{Th_KW}).
In particular for $T\geq \max\lbrace\lambda_{j} : 1\leq j\leq N \rbrace$,
the Fourier Transform of $\mu_{\mathcal{S},T}$ is
\begin{multline*}
\hat{\mu}_{\mathcal{S},T}(\xi) = e^{-im_{\mathcal{S}}\xi}
\int_{A'_{T}}\exp\left(2i\xi \sum_{\gamma\in\mathcal{Z}_{\mathcal{S}}^{*}(T)-\lbrace\lambda_{1},\ldots,\lambda_{N}\rbrace}\re\left(M(\gamma)\frac{e^{2i\pi t_{\gamma}}}{\beta_{\mathcal{S},0}+i\gamma}\right) \right)\diff \omega_{A_{T}}(t) \\
\int_{A(\lambda_{1},\ldots,\lambda_{N})}\exp\left(2i\xi \sum_{j=1}^{N}\re\left(M(\lambda_{j})\frac{e^{2i\pi \theta_{j}}}{\beta_{\mathcal{S},0}+i\lambda_{j}}\right) \right)\diff\omega_{A(\lambda_{1},\ldots,\lambda_{N})}(\theta).
\end{multline*}
We deduce that
\begin{align}\label{Ineq_Fourier_some_selfsuff}
\lvert\hat{\mu}_{\mathcal{S},T}(\xi) \rvert \leq \left\lvert 
\int_{A(\lambda_{1},\ldots,\lambda_{N})}\exp\left(2i\xi \sum_{j=1}^{N}\left\lvert\frac{M(\lambda_{j})}{\beta_{\mathcal{S},0}+i\lambda_{j}}\right\rvert \cos(2\pi \theta_{j} + w_{j}) \right)\diff\omega_{A(\lambda_{1},\ldots,\lambda_{N})}(\theta)
\right\rvert
\end{align}
where $w_{j}$ is defined by $\frac{M(\lambda_{j})}{\beta_{\mathcal{S},0}+i\lambda_{j}} = \left\lvert \frac{M(\lambda_{j})}{\beta_{\mathcal{S},0}+i\lambda_{j}}\right\rvert e^{iw_{j}}$.
The right hand side of inequality~(\ref{Ineq_Fourier_some_selfsuff}) does not depend on $T$, hence for $\xi$ fixed, we can let $T\rightarrow\infty$ and obtain the same inequality for $\hat{\mu}_{\mathcal{S}}(\xi)$.

Thus we only need to prove that the right hand side of (\ref{Ineq_Fourier_some_selfsuff}) 
approaches $0$ when $\lvert\xi\rvert\rightarrow\infty$ to ensure condition~(\ref{Cond_Wiener}).
The function
$$\phi : \left\lbrace
\begin{array}{l l}
A(\lambda_{1},\ldots,\lambda_{N}) & \rightarrow \mathbf{R} \\
(\theta_{1},\ldots,\theta_{N}) &\mapsto 2\sum_{j=1}^{N}\left\lvert\frac{M(\lambda_{j})}{\beta_{\mathcal{S},0}+i\lambda_{j}}\right\rvert \cos(2\pi \theta_{j} + w_{j})
\end{array}
\right.$$
is a non-constant analytic function on a compact set (this is a particular case of Lemma~\ref{lem_NonConstant} below).
Hence there exists $K\geq 1$ such that for each $\theta\in A(\lambda_{1},\ldots,\lambda_{N})$
there exists a multi-index $\underline{k}$ of length $1\leq \lvert \underline{k} \rvert \leq K$
satisfying $\partial^{\underline{k}}\phi(\theta) \neq 0$.
By \cite[VIII 2.2 Prop. 5]{Stein}, one has
\begin{align*}
\lvert\hat{\mu}_{\mathcal{S}}(\xi) \rvert \leq \left\lvert 
\int_{\mathbf{T}(\lambda_{1},\ldots,\lambda_{N})}\exp\left(i\xi 2\sum_{j=1}^{N}\left\lvert\frac{M(\lambda_{j})}{\beta_{\mathcal{S},0}+i\lambda_{j}}\right\rvert \cos(2\pi \theta_{j} + w_{j}) \right)\diff\omega_{\mathbf{T}(\lambda_{1},\ldots,\lambda_{N})}(\theta)
\right\rvert \ll \lvert\xi\rvert^{-1/K}
\end{align*}
with an implicit constant depending on $K$, $\phi$, and on a choice of partition of unity adapted to $\phi$. 
In particular condition~(\ref{Cond_Wiener}) holds, hence $\mu_{\mathcal{S}}$ is continuous.
\end{proof}

We finish this part with the statement of the following general result that we used in the proof of Theorem~\ref{Th_SomeSelfSufficience}.
\begin{lem}\label{lem_NonConstant}
	Let $N\geq 1$ be an integer, $\lambda_{1},\ldots,\lambda_{N} >0$ be distinct real numbers,
	and let $A(\lambda_{1},\ldots,\lambda_{N})$ be the the sub-torus of $\mathbf{T}^{N}$ given by
	the topological closure of the $1$-parameter group 
	$\left\lbrace y \left(\lambda_{1},\ldots,\lambda_{N}\right) : y\in\mathbf{R} \right\rbrace/\mathbf{Z}^{N}.$
	For every $(M_{j})_{1\leq j\leq N}, (w_{j})_{1\leq j\leq N} \in \mathbf{C}^{N}$, such that for at least one $1\leq j\leq N$, one has  $M_{j} \neq 0$,
	one has that
	$$\phi_{M,w} : \left\lbrace
	\begin{array}{l l}
	A(\lambda_{1},\ldots,\lambda_{N}) & \rightarrow \mathbf{R} \\
	(\theta_{1},\ldots,\theta_{N}) &\mapsto \sum_{j=1}^{N}M_{j}\cos(2\pi \theta_{j} + w_{j})
	\end{array}
	\right.$$
	is a non-constant analytic function.
\end{lem}

\begin{proof}
	The function $\phi$ is analytic as a finite sum of analytic functions.
	It is defined as a linear combination of characters of $A(\lambda_{1},\ldots,\lambda_{N})$ seen as a subset of $\mathbf{T}^{N}$.
	Precisely, for $1\leq j\leq N$ we denote by $\chi_{j}$ (resp. $\chi_{j}'$) the character
	$A(\lambda_{1},\ldots,\lambda_{N}) \ni  (\theta_{1},\ldots,\theta_{N})   \mapsto e^{2i\pi \theta_{j}}$ (resp. $ \mapsto e^{-2i\pi \theta_{j}}$).
	Let us study the values on $\left\lbrace y \left(\lambda_{1},\ldots,\lambda_{N}\right) : y\in\mathbf{R} \right\rbrace/\mathbf{Z}^{N}$, 
	and take the derivative at $y=0$, we obtain distinct values $\lambda_{j}, -\lambda_{j}$, all of them distinct from $0$.
	Hence the $2N+1$ characters $1, \chi_{1},\chi_{1}',\ldots,\chi_{N},\chi_{N}'$ are distinct.
	By a result of Dedekind--Artin \cite[VI, Th. 4.1]{LangAlg}, the characters are linearly independent over $\mathbf{C}$.
	Given $(M_{j})_{1\leq j\leq N}, (w_{j})_{1\leq j\leq N} \in \mathbf{C}^{N}$,
	one has that
	$$\phi_{M,w}  = \sum_{j=1}^{N} \frac{M_{j}}{2}\left( e^{iw_{j}}\chi_{j} + e^{-iw_{j}}\chi_{j}' \right).
	$$
	is linearly independent from $1$, as soon as one $M_{j}$ is non-zero.
	In particular the function $\phi_{M,w}$ is non constant, and the proof is complete.
\end{proof}

\subsection{Symmetry}\label{sub_Sym}

We prove Theorem~\ref{Th_Indep_Sym} by showing that under its hypotheses, for every $T$, 
the distribution $\mu_{\mathcal{S},T}$ is symmetric with respect to its mean.
For this we use the Kronecker--Weyl Theorem (Theorem \ref{Th_KW}),
and the following result.

\begin{lem}\label{Lem_Eq_sym}
	The following assertions are equivalent:
	\begin{enumerate}
		\item for all integral linear combination  $\sum\limits_{\gamma\in\mathcal{Z}_{\mathcal{S}}}k_{\gamma}\gamma = 0$, $k_{\gamma}\in\mathbf{Z}$, one has $\sum\limits_{\gamma\in\mathcal{Z}_{\mathcal{S}}}k_{\gamma}\equiv 0\ [\bmod\ 2]$,
		\item For every finite subset $\lbrace \gamma_{1},\ldots\gamma_{N} \rbrace \subset \mathcal{Z}_{\mathcal{S}}$, the element $\left(\frac{1}{2},\ldots,\frac{1}{2}\right)$ is in the closure of the one parameter group 
		$\lbrace (\gamma_{1}y,\ldots,\gamma_{N}y) : y\in \mathbf{R}\rbrace/\mathbf{Z}^{N}$ in $\mathbf{T}^{N}$.
	\end{enumerate}
\end{lem}
\begin{proof}
	This is a consequence of Lemma~\ref{Lem_AA=Closure}, 
	one has $\left(\frac{1}{2},\ldots,\frac{1}{2}\right)\in A(\gamma_1, \ldots, \gamma_N)$ if and only if for every $k\in A(\gamma)^{\bot}$ one has $\sum\limits_{i=1}^{N}k_{i} =0 \ [\bmod\ 2]$
\end{proof}

\begin{Rk}
In the formulation of Lemma~\ref{Lem_Eq_sym}, LI is equivalent to the fact that the closure of the one parameter group generated by a finite number of ordinates is always the largest possible (i.e. the $N$-dimensional torus when there are $N$ ordinates).
The improvement in Theorem~\ref{Th_Indep_Sym}, is that we only need to know that the element $\left(\frac{1}{2},\ldots,\frac{1}{2}\right)$ is in this group to obtain the symmetry.
\end{Rk}

\begin{proof}[Proof of Theorem~\ref{Th_Indep_Sym}]
By Theorem \ref{Th_KW} and Lemma~\ref{Lem_Eq_sym}, we deduce that for all $T$ large enough, one has $A_{T} = A_{T} + (\frac{1}{2},\ldots,\frac{1}{2})$.
This way we can change variables $a \rightarrow a + (\frac{1}{2},\ldots,\frac{1}{2})$ in the integral defining $\mu_{\mathcal{S},T}$.
For every $T>1$, and for every bounded Lipschitz continous function $g$ one has
\begin{align*}
\int_{\mathbf{R}}g(t)\diff\mu_{\mathcal{S},T} &=
\int_{A_{T}}\tilde{g}(a)\diff\omega_{A_{T}} \\
&= \int_{A_{T}}\tilde{g}\left(a+ \left(\frac{1}{2},\ldots,\frac{1}{2}\right)\right)\diff\omega_{A_{T}}
= \int_{\mathbf{R}}g(2m_{f}- t)\diff\mu_{\mathcal{S},T} 
\end{align*}
where we use the definition of $\tilde{g}$ given in (\ref{form_tilde}). 
One has $\tilde{g}(a+ (\frac{1}{2},\ldots,\frac{1}{2})) = \tilde{h}(a)$ where $h$ is the function given by $h(x)=g(2m_{f} -x)$. 
Take $T$ arbitrarily large: by (\ref{Formule_BoundsLimDist}) the property of symmetry is true for $\mu_{\mathcal{S}}$.
\end{proof}

\begin{Rk}
In particular under the condition of Theorem~\ref{Th_Indep_Sym} for the set of the zeros of maximal real part associated to a set $\mathcal{S}$ of $L$-functions, we deduce that if the prime number race associated to $\mathcal{S}$ is biased it implies that the average value is not $0$.
So if the prime number race is biased, either RH is satisfied, or at least one of the $L$-functions vanishes at a point of $[\frac{1}{2},1)$.
\end{Rk}

\subsection{Riemann Hypothesis and support}\label{sub_Support}

We prove Theorem~\ref{Th_withRH}, the first point is a generalization of \cite[Th. 1.2]{RS}.

\begin{proof}[Proof of Theorem~\ref{Th_withRH}(\ref{Th1_RH})]
The proof follows from an adaptation of \cite[2.2]{RS}.
The idea is to find a lower bound for the measure of the set $\lbrace y\leq M : E_{\mathcal{S}}(e^{y})\geq A\rbrace$ as $M$ varies.
One can also see \cite[Th. 3.40]{L_These} for a very detailed proof. 
\end{proof}

In the case RH is not satisfied ($\beta_{\mathcal{S},0}> 1/2$), 
one can conjecture that $\mathcal{Z}_{\mathcal{S}}$ is not too large.
In particular it may have density equal to $0$ (in the set of all zeros).
This is the point of Zero Density Theorems, and in particular of the condition in Theorem~\ref{Th_withRH}(\ref{Th2_nRH}).

\begin{proof}[Proof of Theorem~\ref{Th_withRH}(\ref{Th2_nRH})]
	We assume that the sum $\sum_{\gamma \in \mathcal{Z}_{\mathcal{S}}}\frac{1}{\lvert \beta_{\mathcal{S},0} + i\gamma\rvert}$ converges,
	hence for every $T$, the limiting logarithmic distribution $\mu_{\mathcal{S},T}$ has compact support included in the interval 
	$$\left[m_{\mathcal{S}}- \sum_{\gamma \in \mathcal{Z}_{\mathcal{S}}}\frac{1}{\lvert \beta_{\mathcal{S},0} + i\gamma\rvert} , m_{\mathcal{S}}+ \sum_{\gamma \in \mathcal{Z}_{\mathcal{S}}}\frac{1}{\lvert \beta_{\mathcal{S},0} + i\gamma\rvert}\right].$$
	By Proposition \ref{Prop_LimOfDist}, $\mu_{\mathcal{S}}$ has compact support included in the same interval.	
\end{proof}

\begin{Rk}
Theorem~\ref{Th_withRH}(\ref{Th2_nRH}) could indicate a way to find completely biased prime number races: 
in the case one has an $L$-function with a zero $\beta_{0}$ in the interval $\left(\frac{1}{2},1\right)$ (e.g. a Siegel zero), 
and such that it has no zeros of larger real part,
we can imagine that there will not be many other zeros of maximal real part (if ever they exist). 
For example if
\begin{align*}
\sum_{\gamma \in\mathcal{Z}_{L}}\frac{1}{\lvert \beta_{0}+i\gamma\rvert} \leq m(L,\beta_{0})
\end{align*}
then we would have $\delta(L) = 0$.
But the existence of such an $L$-function seems very unlikely.
\end{Rk}

\subsection{Concentration of the distribution}\label{sub_ChebInequality}

In the process of looking for large biases, a strategy is to ensure that the distribution is 
concentrated around its average value (the average value being itself far from zero).
Such concentration results are usually consequences of Chebyshev's inequality (e.g. \cite[Prop. (1.2)]{Bass}).

\begin{lem}[Chebyshev's inequality]\label{Lem_ChebyshevIneq}
	Let $X$ be a random variable with average $\mathbb{E}(X)$ and finite variance $\Var(X)$.
	For any $\alpha >0$, one has:
	\begin{align}\label{ChebyshevInequality}
	\mathbf{P}(\lvert X - \mathbb{E}(X) \rvert \geq \alpha)\leq \frac{\Var(X)}{\alpha^{2}}.
	\end{align}
\end{lem}

Using Theorem \ref{Th_DistLim}, we can compute the average and variance of a random variable that has distribution equal to $\mu_{\mathcal{S}}$.
With upper bounds on the variance or lower bounds for the average, 
we may deduce concentration results (hence bounds for the bias).

\begin{cor}\label{Cor_ChebyshevIneq}
	\begin{itemize}
		\item In the case $m_{\mathcal{S}}<0$ one has:
		$
		\overline{\delta}(\mathcal{S}) \leq \frac{2}{m_{\mathcal{S}}^{2}}\sum_{\gamma \in \mathcal{Z}_{\mathcal{S}}^{*}}\frac{\lvert M(\gamma)\rvert^{2}}{(\beta_{\mathcal{S},0}^{2}+\gamma^{2})}.
		$
		\item In the case $m_{\mathcal{S}}>0$ one has:
		$
		\underline{\delta}(\mathcal{S}) \geq 1- \frac{2}{m_{\mathcal{S}}^{2}}\sum_{\gamma \in \mathcal{Z}_{\mathcal{S}}^{*}}\frac{\lvert M(\gamma)\rvert^{2}}{(\beta_{\mathcal{S},0}^{2}+\gamma^{2})}.
		$
	\end{itemize}
\end{cor}

In particular if the ratio $\frac{2}{m_{\mathcal{S}}^{2}}\sum_{\gamma \in \mathcal{Z}_{\mathcal{S}}^{*}}\frac{\lvert M(\gamma)\rvert^{2}}{(\beta_{\mathcal{S},0}^{2}+\gamma^{2})}$ is small ($< 1/2$),
we obtain the second version of Corollary~\ref{Cor_mean_bias}.

\begin{proof}
	The proof is inspired from \cite[Lem. 2.7]{FioEC}.
	Assume $m_{\mathcal{S}}<0$,
	one has
	\begin{align*}
	\overline{\delta}(\mathcal{S}) = \limsup\frac{1}{Y}\int_{2}^{Y}\mathbf{1}_{\geq 0}(E_{\mathcal{S}}(e^{y}))\diff y 
	\leq \limsup\frac{1}{Y}\int_{2}^{Y}g_{n}(E_{\mathcal{S}}(e^{y}))\diff y
	\end{align*}
	where for $n\in\mathbf{N}$, the function $g_{n}$ is continuous $n$-Lipschitz, bounded and has support $\left[\frac{-1}{n},+\infty\right)$. 
	Therefore one has
	\begin{align*}
	\limsup\frac{1}{Y}\int_{2}^{Y}g_{n}(E_{\mathcal{S}}(e^{y}))\diff y = \mu_{\mathcal{S}}(g_{n}) 
	\leq \mu_{\mathcal{S}}\left[\frac{-1}{n},+\infty\right)  = \mathbf{P}\left(X \geq \frac{-1}{n}\right)
	\end{align*}
	for $X$ a random variable of law $\mu_{\mathcal{S}}$.
	For $n$ large enough, so that $m_{\mathcal{S}} +\frac{1}{n} <0$,
	we apply Chebyshev's inequality
	\begin{align*}
	\mathbf{P}\left(X \geq \frac{-1}{n}\right) 
	\leq \mathbf{P}\left(\lvert\mathbb{E}(X)-X \rvert \geq \left\lvert m_{\mathcal{S}} +\frac{1}{n}\right\rvert\right)\leq \frac{\Var(X)}{\left\lvert m_{\mathcal{S}} +\frac{1}{n}\right\rvert^{2}}.
	\end{align*}
	Letting $n\rightarrow \infty$ and using the values obtained in Theorem~\ref{Th_DistLim} for the average value and the variance yields the result.
	The case $m_{\mathcal{S}}>0$ follows from similar computations.
\end{proof}